\documentclass[a4paper,10pt
]{amsart}
\usepackage{amsmath,amssymb,amsthm}

\usepackage[alphabetic, abbrev]{amsrefs}
\usepackage{enumerate
}
\usepackage{hyperref}
\usepackage{tikz-cd}
\usepackage{tikz}

\numberwithin{equation}{section}
\newtheorem{theorem}[subsection]{Theorem}
\newtheorem{corollary}[subsection]{Corollary}
\newtheorem{lemma}[subsection]{Lemma}
\newtheorem{proposition}[subsection]{Proposition}
\theoremstyle{definition}
\newtheorem{definition}[subsection]{Definition}
\newtheorem{remark}[subsection]{Remark}
\newtheorem{example}[subsection]{Example}

\newtheorem{construction}[subsection]{Construction}

\newcommand{\cal}{\mathcal}

\DeclareMathOperator{\Mod}{Mod}

\newcommand{\arxivlink}[1]{\href{http://arxiv.org/abs/#1}{\texttt{arXiv:#1}}}

\title[On log ${\rm TAQ}$ and log ${\rm THH}$]{On the relationship between  \\ logarithmic TAQ and logarithmic THH}
\author{Tommy Lundemo}
\address{IMAPP, Radboud University Nijmegen, The Netherlands}
\email{t.lundemo@math.ru.nl}
\date{\today}

\begin{document}
\maketitle

\begin{abstract} We provide a new description of logarithmic topological Andr\'e--Quillen homology in terms of the indecomposables of an augmented ring spectrum. The new description allows us to interpret logarithmic TAQ as an abstract cotangent complex, and leads to a base-change formula for logarithmic topological Hochschild homology. The latter is analogous to results of Weibel--Geller for Hochschild homology of discrete rings, and of McCarthy--Minasian and Mathew for topological Hochschild homology. For example, our results imply that logarithmic THH satisfies base-change for tamely ramified extensions of discrete valuation rings. 
\end{abstract}

\section{Introduction} Ramification is a notion in ordinary algebra which has no clear generalization to the context of ${\Bbb E}_{\infty}$-rings. For example, there is evidence that the complexification map ${\rm KO} \to {\rm KU}$ relating the real and complex periodic $K$-theory spectra should be regarded an unramified extension. It fails to be an \'etale map in the sense of Lurie \cite[Chapter 7]{Lur}, as the map of graded rings \[\pi_0({\rm KU}) \otimes_{{\pi}_0({\rm KO})} \pi_*({\rm KO}) \xrightarrow{} \pi_*({\rm KU})\] fails to be an isomorphism. Nonetheless, it enjoys a strong formal \'etaleness property in that the unit map \[{\rm KU} \to {\rm THH}^{\rm KO}({\rm KU})\] from ${\rm KU}$ to its topological Hochschild homology relative to ${\rm KO}$ is a stable equivalence. For ordinary Hochschild homology of discrete rings, the analogous property is satisfied for any \'etale map \cite{GW91}. 

In general, suppose that $R \to A$ is a map of ${\Bbb E}_{\infty}$-rings such that the unit map \[A \to {\rm THH}^R(A)\] from $A$ to its topological Hochschild homology relative to $R$ is a stable equivalence. This implies that the topological Andr\'e--Quillen homology ${\rm TAQ}^R(A)$, the $A$-module corepresenting derivations, is contractible. As we review below, these formal \'etaleness properties are, up to a mild finiteness condition, equivalent to being \'etale once attention is restricted to \emph{connective} ring spectra.

\subsection{Examples of ramified extensions of ${\Bbb E}_{\infty}$-rings} Notions of formal \'etaleness are less useful if one wants to distinguish between tame and wild ramification. We give an example of a map of ${\Bbb E}_{\infty}$-rings which we think of as tamely ramified, and one which we think of as wildly ramified. Both are maps of connective ${\Bbb E}_{\infty}$-rings that fail to be \'etale, and so neither map satsifies any of the formal \'etaleness properties discussed above. 

Let $p$ be an odd prime and consider the inclusion ${\ell}_p \to {\rm ku}_p$ of the $p$-complete connective Adams summand. On homotopy rings, this is the map ${\Bbb Z}_p[v] \to {\Bbb Z}_p[u]$ sending $v$ to $u^{p -1}$. Thinking of the coefficients $v$ and $u$ as uniformizers, this is reminiscent of a tamely ramified extension of number rings. As the ring spectra involved are connective, one might expect that the inclusion of the Adams summand should be regarded as tamely ramified. 

Consider now the complexification map ${\rm ko} \to {\rm ku}$ relating the real and complex \emph{connective} $K$-theory spectra. The map of homotopy rings is more complicated. However, restricting attention to the behavior of the Bott classes, the fact that the Bott class $w \in \pi_8({\rm ko})$ maps to $u^4$ leads us to expect that the complexification map is wildly ramified at the prime $2$.

As pointed out by Rognes \cite[Remark 7.3]{Rog14}, there is evidence for both of the above expectations through a natural interpretation of Noether's theorem relating tame ramification for ordinary rings to the existence of a normal basis. However, it is not clear how the tame and wild ramification of these maps of ${\Bbb E}_{\infty}$-rings should be interpreted in terms of ${\rm TAQ}$ and ${\rm THH}$. The following idea from logarithmic geometry has proven useful in addressing this issue:  rigidifying schemes with the extra data of a \emph{logarithmic structure} allows one to treat some tamely ramified extensions as if they were unramified. For example, a finite extension of complete discrete valuation rings in mixed characteristic it is at most tamely ramified if and only if it is \emph{log \'etale}. The essential property enjoyed by the tamely ramified extensions is the vanishing of a certain module of logarithmic differentials. These generalize the classical K\"ahler differentials $\Omega^1_{C | B}$ associated to a map $B \to C$ of discrete rings, which vanish for any \'etale map. The converse fails to hold: For example, any map of perfect ${\Bbb F}_p$-algebras is formally \'etale.   \\

In this paper we generalize results relating various notions of formal \'etaleness from ${\Bbb E}_{\infty}$-ring spectra to \emph{logarithmic ring spectra}. The latter for instance arise as approximations of periodic ring spectra that are finer than the connective covers, but nevertheless retain some characteristics enjoyed by connective ring spectra. This generalization is desirable for at least two reasons: first, it sheds light on the nature of tame and wild ramification in the context of ${\Bbb E}_{\infty}$-rings. For example, we will see below that the unit map from ${\rm ku}_p$ to the \emph{logarithmic} ${\rm THH}$ of ${\rm ku}_p$ relative to the Adams summand is a stable equivalence, while the corresponding statement fails for the complexification map ${\rm ko} \to {\rm ku}$. Secondly, it relaxes the necessary connectivity hypotheses that appear in results relating formal \'etaleness properties for ordinary ${\Bbb E}_{\infty}$-rings.

The core ingredient in this generalization is a new description of logarithmic topological Andr\'e--Quillen homology, which was initially introduced by Rognes \cite{Rog09} and later studied by Sagave \cite{Sag14}. The new description involves Rognes' notion of \emph{repletion} in a manner reminiscent of how it is employed in the definition of logarithmic ${\rm THH}$ of Rognes, Sagave and Schlichtkrull \cite{RSS15}; see also Remark \ref{rem:logkahler} for an algebro-geometric interpretation of the new description. Our formulation is simultaneously more reminiscent of the definition of ordinary ${\rm TAQ}$ of Basterra \cite{Bas99} in that it arises as the module of indecomposables of an augmented ring spectrum. It therefore allows us to describe the relationship between log ${\rm TAQ}$ and log ${\rm THH}$ in a manner largely analogous to the non-logarithmic case.

\subsection{The \'etale descent formula for ${\rm THH}$} To put our results in context, we first review the results which we aim to generalize.  Let $f \colon R \to A$ be a map of ${\Bbb E}_{\infty}$-ring spectra. We say that $f$ 

\begin{enumerate}
\setcounter{enumi}{-1}
\item is \emph{\'etale} is the map $\pi_0(R) \to \pi_0(A)$ of discrete commutative rings is \'etale and the map of graded rings $\pi_0(A) \otimes_{\pi_0(R)} \pi_*(R) \xrightarrow{} \pi_*(A)$ is an isomorphism;
\item satisfies \emph{\'etale descent} if the canonical map $A \wedge_R {\rm THH}(R) \xrightarrow{} {\rm THH}(A)$ is an equivalence;   
\item is \emph{formally ${\rm THH}$-\'etale} if the unit map $A \xrightarrow{} {\rm THH}^R(A)$ is an equivalence;
\item is \emph{formally ${\rm TAQ}$-\'etale} if the $A$-module spectrum ${\rm TAQ}^R(A)$ is contractible. 
\end{enumerate}

The above properties go by various names in the literature; we choose to follow the terminology of Richter \cite[Definitions 8.3 and 8.8]{Ric17}. These notions are for instance useful for studying morphisms which are not necessarily \'etale, but exhibit similar formal properties. The most prominent class of examples is perhaps the faithful Galois extensions of Rognes \cite{Rog08}.

Properties (0) and (3) are closely related, since the topological Andr\'e--Quillen homology ${\rm TAQ}^R(A)$ is a model for the cotangent complex ${\Bbb L}_{A | R}$. For maps between connective ring spectra, formally ${\rm TAQ}$-\'etale implies \'etale as soon as $\pi_0(A)$ is finitely presented over $\pi_0(R)$ by \cite[Lemma 8.9]{Lur11}. It is always the case that \'etale morphisms are formally ${\rm TAQ}$-\'etale by \cite[Corollary 7.5.4.5]{Lur}. We now present an alternative proof of this fact, which highlights how ${\rm THH}$ naturally fits into this story.

Let $B \to C$ be an \'etale morphism of discrete commutative rings. Weibel--Geller \cite{GW91} prove that Hochschild homology satisfies the base-change property \[{\rm HH}_*(C) \cong C \otimes_B {\rm HH}_*(B),\] and they relate this to descent for Hochschild homology along $B \to C$; hence the name \emph{\'etale descent} for this property. Mathew \cite{Mat17} generalizes the \'etale descent formula of Weibel--Geller to ${\Bbb E}_{\infty}$-ring spectra: if $f \colon R \to A$ is \'etale, then it satisfies \'etale descent. It is largely formal (cf.\ Section \ref{etaledescentimpliesthhetale}) to see that this implies that $f$ is formally ${\rm THH}$-\'etale, and an argument due to Rognes \cite[Lemma 9.4.4]{Rog08} based on work of Basterra and Mandell \cite{BM05} applies to show that $f$ is formally ${\rm TAQ}$-\'etale in this case.  

For morphisms between \emph{connective} ${\Bbb E}_{\infty}$-ring spectra, one finds that the \'etale descent formula holds as soon as ${\rm TAQ}^R(A)$ is contractible, since \cite[Proof of Proposition 7.5.1.15]{Lur} exhibits $R \to A$ as what Mathew calls \emph{strongly $0$-cotruncated} in this case. Mathew shows that any such morphism satisfies \'etale descent, and together this shows:

\begin{theorem}\label{etaledescentthh} Let $f \colon R \to A$ be a formally ${\rm TAQ}$-\'etale morphism between connective ${\Bbb E}_{\infty}$-ring spectra. Then $f$ satisfies \'etale descent, i.e.\ the canonical map \[A \wedge_R {\rm THH}(R) \xrightarrow{} {\rm THH}(A)\] is a stable equivalence.
\end{theorem}

We remark that this strengthens the \'etale descent formula of McCarthy--Minasian \cite[Section 5]{MM03}, where the authors impose a finiteness hypothesis roughly amounting to $A$ being perfect as an $R$-module. 

From the above discussion, we have that the following implications always hold true: \[\text{\'etale descent} \implies \text{formally THH-\'etale} \implies \text{formally TAQ-\'etale}.\] Moreover, the content of Theorem \ref{etaledescentthh} is that the converse statements hold as soon as the ring spectra involved are connective. Both reverse implications are known to fail in general, the first due to Mathew \cite{Mat17} and the second due to McCarthy--Minasian \cite{MM03}. One of the main results of the present paper, Theorem \ref{mainthm}, is an analogue of Theorem \ref{etaledescentthh} for logarithmic ring spectra.

\subsection{Logarithmic ring spectra} A \emph{pre-logarithmic ring} $(R, P, \beta)$ consists of a discrete commutative ring $R$, a commutative monoid $P$ and a map $\beta \colon P \to (R, \cdot)$ of commutative monoids to the underlying multiplicative monoid of $R$. A map $(f, f^\flat) \colon (R, P, \beta) \to (A, M, \alpha)$ of pre-log rings consists of a map $f^\flat \colon P \to M$ of commutative monoids and a map $f \colon R \to A$ of commutative rings such that $\alpha \circ f^\flat = (f, \cdot) \circ \beta$. A pre-log ring $(A, M, \alpha)$ determines a localization $A[M^{-1}]$, and we think of the pre-log ring itself as an intermediate localization between $A$ and $A[M^{-1}]$.

\emph{Pre-logarithmic ring spectra} can be defined using \emph{commutative ${\cal J}$-space monoids} as introduced by Sagave and Schlichtkrull \cite{SS12}. These can be thought of as $QS^0$-graded ${\Bbb E}_{\infty}$-spaces. Every commutative (symmetric) ring spectrum $A$ gives rise to a commutative $\cal J$-space monoid $\Omega^{\cal J}(A)$, which we think of the underlying multiplicative graded monoid of $A$. Examples of pre-logarithmic ring spectra arise via homotopy classes: for example, the Bott class $u \in \pi_2({\rm ku})$ gives rise to a pre-log structure $D(u) \to \Omega^{\cal J}({\rm ku})$. The localization it determines is the periodic theory ${\rm KU}$. We refer to Sections \ref{commjspace} and \ref{logringspectra} of the present paper and \cites{Rog09, SS12, Sag14, RSS15} for more detailed introductions to logarithmic ring spectra.  

\subsection{The \'etale descent formula for logarithmic ${\rm THH}$}  We now aim to generalize Theorem \ref{etaledescentthh} to the context of pre-logarithmic ring spectra. For this we shall make use of logarithmic ${\rm TAQ}$ as developed in \cites{Rog09, Sag14} and the present paper, and logarithmic ${\rm THH}$ as developed in \cites{Rog09, RSS15}. 

We can make definitions analogous to those already discussed for ${\rm TAQ}$ and ${\rm THH}$ in the context of pre-logarithmic ring spectra.

\begin{definition}\label{def:formallylogetale} Let $(f, f^\flat) \colon (R, P) \xrightarrow{} (A, M)$ be a map of pre-logarithmic ring spectra. The morphism $(f, f^\flat)$ 

\begin{enumerate}
\item satisfies \emph{log \'etale descent} the natural map $A \wedge_R {\rm THH}(R, P) \xrightarrow{} {\rm THH}(A, M)$ is a stable equivalence;
\item is \emph{formally log {\rm THH}-\'etale} if the unit map $A \xrightarrow{} {\rm THH}^{(R, P)}(A, M)$ is a stable equivalence;
\item is \emph{formally log ${\rm TAQ}$-\'etale} if the $A$-module spectrum ${\rm TAQ}^{(R, P)}(A, M)$ is contractible.
\end{enumerate} 
\end{definition}

The following generalizes Theorem \ref{etaledescentthh} to pre-logarithmic ring spectra: 

\begin{theorem}\label{mainthm} Let $(R, P) \to (A, M)$ be a map of pre-logarithmic ring spectra. Then the following implications always hold true: \[\text{log \'etale descent} \implies \text{formally log $\rm THH$-\'etale} \implies \text{formally log $\rm TAQ$-\'etale}.\] Moreover, if the pre-logarithmic ring spectra involved are connective, then the reverse implications hold. 
\end{theorem}

In the theorem, \emph{connectivity} means that both the underlying commutative ring spectrum and a graded analogue of the spherical monoid ring of the underlying graded ${\Bbb E}_\infty$-space are connective. We make this precise in Definition \ref{logconnective}.

We remark that our terminology differs from that used used in \cite{RSS18}, where the term ``formally log ${\rm THH}$-\'etale" is used for what we have called ``log \'etale descent." With our terminology, the main result of \cite[Section 6]{RSS18} reads as follows:

\begin{theorem}[Rognes--Sagave--Schlichtkrull] Let $p$ be an odd prime and let ${\rm ku}_{p}$ be the $p$-complete connective complex $K$-theory spectrum. The inclusion of the Adams summand $\ell_p$ induces a map $(\ell_p, D(v)) \to ({\rm ku}_{p}, D(u))$ of pre-logarithmic ring spectra which satisfies log \'etale descent. 
\end{theorem}

As a corollary, the results of \cite{RSS18} combined with Theorem \ref{mainthm} provide the following description of the relative logarithmic topological Hochschild homology ${\rm THH}^{(\ell_p, D(v))}({\rm ku}_{p}, D(u))$:

\begin{corollary}\label{thhku} The morphism $(\ell_p, D(v)) \xrightarrow{} ({\rm ku}_{p}, D(u))$ is formally log ${\rm THH}$-\'etale; that is, the unit map \[{\rm ku}_{p} \xrightarrow{\simeq} {\rm THH}^{(\ell_p, D(v))}({\rm ku}_{p}, D(u))\] is a stable equivalence of commutative symmetric ring spectra.  \qed
\end{corollary}

In particular, the inclusion of the Adams summand satisfies all three properties in Definition \ref{def:formallylogetale}. As a map of connective ring spectra that fails to be \'etale, the inclusion of the Adams summand does not satisfy any of the analogous properties for ordinary ${\rm THH}$ and ${\rm TAQ}$.

The complexification map ${\rm ko} \to {\rm ku}$ participates in a map $({\rm ko}, D(w)) \to ({\rm ku}, D(u))$ of pre-logarithmic ring spectra that realizes the periodic complexification map ${\rm KO} \to {\rm KU}$ on localizations. H\"oning--Richter prove that the associated logarithmic ${\rm TAQ}$ is \emph{not} contractible \cite[Theorem 3.2]{HR21}. Consequently, Theorem \ref{mainthm} implies that this map fails to satisfy any of the formal log \'etaleness properties discussed here, providing further evidence that the complexification map should be regarded a wildly ramified extension.

\subsection{Log \'etale descent for tamely ramified extensions} In the context of discrete log rings, examples of log \'etale maps arise as tamely ramified extensions of discrete valuation rings. A discrete valuation ring $R$ always has a natural log structure given by the inclusion of the non-zero elements $R \cap {\rm GL}_1(K)$ of $R$, where $K$ denotes the fraction field of $R$. In Section \ref{sec:discrete} we relate the usual notion of log \'etaleness in the sense of Kato \cite{Kat89} to formal \'etaleness properties in such a way that Theorem \ref{mainthm} applies to obtain:

\begin{theorem}[Theorem \ref{thm:logthhbasechangedvrs}] Let $R \to A$ be a tamely ramified finite extension of complete discrete valuation rings in mixed characteristic $(0, p)$ with perfect residue fields. Let $K \to L$ be the induced map of fraction fields. Then the induced map $(R, R \cap {\rm GL}_1(K)) \to (A, A \cap {\rm GL}_1(L))$ of log rings satisfies log \'etale descent; that is, the canonical map \[A \wedge_R {\rm THH}(R, R \cap {\rm GL}_1(K)) \to {\rm THH}(A, A \cap {\rm GL}_1(L))\] is a stable equivalence.
\end{theorem}

\subsection{Logarithmic topological Andr\'e--Quillen homology} Let $A$ be a commutative ring spectrum and let $A \to B \to A$ be an augmented commutative $A$-algebra. A homotopically meaningful construction of the indecomposables was introduced and studied by Basterra \cite{Bas99}, and gives rise to a functor \[{\rm taq}^A \colon {\rm CAlg}_{A // A} \to {\rm Mod}_A\] from the category of commutative augmented $A$-algebras to that of $A$-modules. The topological Andr\'e--Quillen homology ${\rm TAQ}^R(A)$ associated to a map of commutative ring spectra $R \to A$ is by definition the $A$-module ${\rm taq}^A(A \wedge_R A)$. 

For a morphism $(R, P) \to (A, M)$ of pre-logarithmic ring spectra, Sagave \cite{Sag14} constructs an $A$-module $\widetilde{{\rm TAQ}}{}^{(R, P)}(A, M)$ which corepresents \emph{logarithmic derivations}, which we review in Section \ref{logtaq} of the present paper. The definition of this $A$-module involves certain Segal $\Gamma$-spaces defined in terms of the morphism $P \to M$ of commutative ${\cal J}$-space monoids. We revisit the construction of logarithmic derivations to provide a new formulation of log ${\rm TAQ}$ which is more reminiscent of the non-logarithmic definition:

\begin{theorem}[Definition \ref{def:logtaq}, Theorem \ref{thm:taqcorepresents}]\label{thm:newtaq} The $A$-module $\widetilde{{\rm TAQ}}{}^{(R, P)}(A, M)$ is naturally weakly equivalent to ${\rm taq}^A(C)$, where $C$ is an explicitly defined augmented commutative $A$-algebra dependent on the morphism $(R, P) \to (A, M)$. 
\end{theorem}

This prompts us to redefine the \emph{logarithmic topological Andr\'e--Quillen homology} ${\rm TAQ}^{(R, P)}(A, M)$ as the $A$-module ${\rm taq}^A(C)$. As we explain in Remark \ref{rem:logkahler}, our definition is analogous to the description of the log K\"ahler differentials as the conormal of the \emph{log diagonal} studied by Kato--Saito \cite{KS04}.

\subsection{Logarithmic ${\rm TAQ}$ as a cotangent complex} From the point of view of Lurie's cotangent complex formalism \cite[Section 7]{Lur}, the results of Basterra--Mandell \cite{BM05} exhibit ${\rm Mod}_A$ as the \emph{tangent category} of the category ${\rm CAlg}$ of commutative algebras at $A$, and topological Andr\'e--Quillen homology ${\rm TAQ}^R(A)$ as its corresponding cotangent complex ${\Bbb L}_{A | R}$.

By definition, the stable category ${\rm Sp}({\rm PreLog}_{(A, M)//(A, M)})$ is the tangent category of the category of pre-logarithmic ring spectra at $(A, M)$. A variant of this category has been proposed as a category of \emph{log modules} ${\rm Mod}_{(A, M)}$ in \cite[Remark 8.8]{Rog09}. In Section \ref{sec:cotangentcomplex} we recover an analogue of this category by stabilizing a \emph{replete} model structure ${\rm PreLog}^{\rm rep}_{(A, M)//(A, M)}$. With this model structure, ${\rm TAQ}^{(R, P)}(A, M)$ arises as the $A$-module underlying the $(A, M)$-module ${\Bbb L}^{\rm rep}_{(A, M)|(R, P)}$ obtained by forming the associated cotangent complex.  In addition, the replete model structure gives rise to an interpretation of log ${\rm THH}$ as a cyclic bar construction in the category of pre-log ring spectra. 

\subsection{A logarithmic Andr\'e--Quillen spectral sequence} As part of the proof of Theorem \ref{mainthm}, we apply a theorem due to Kuhn \cite[Theorem 3.10]{Kuh06} to express logarithmic ${\rm THH}$ as the homotopy limit of a tower of fibrations, where the fiber of each map in the tower can be described in terms of logarithmic ${\rm TAQ}$. This gives rise to a conditionally convergent spectral sequence \[E^{s, t}_1 = \pi_{t - s}(\big{[}\bigwedge_A^s \Sigma {\rm TAQ}^{(R, P)}(A, M)\big{]}_{h\Sigma_s}) \implies \pi_{t - s}({\rm THH}^{(R, P)}(A, M)),\] see \cite[IX.4.2]{BK72}, \cite{Boa99} and Section \ref{lastsection} of the present paper. In the case of ordinary ${\rm THH}$, this recovers a version of the spectral sequence constructed by Minasian \cite{Min03}. Related results leading to a version of the logarithmic Andr\'e--Quillen spectral sequence above appear in the unpublished PhD thesis of Franklin \cite{Fra15}. To the best of the author's understanding, the arguments in \emph{loc.\ cit.}\ are highly dependent on the arguments of McCarthy--Minasian \cites{Min03, MM03}, on which the present work does not rely.

\subsection{Outline} In Section \ref{commjspace} we review our conventions and some basic properties of commutative ${\cal J}$-space monoids. In Section \ref{cycbar} we review the cyclic and replete bar constructions, and we interpret both in terms of appropriate simplicial tensors. Section \ref{logringspectra} recalls basic properties of logarithmic ring spectra and the logification construction, while in Section \ref{logthh} we discuss relative logarithmic ${\rm THH}$ and its basic properties. In Section \ref{logtaq} we provide the new description of log ${\rm TAQ}$, while in Section \ref{lastsection} we give the proof of Theorem \ref{mainthm}.  In Section \ref{sec:discrete} we discuss log \'etale descent for tamely ramified extensions of discrete valuation rings, while in Section \ref{sec:cotangentcomplex} we give the cotangent complex interpretation of log ${\rm TAQ}$. 

\subsection{Acknowledgments} The author thanks Stefano Ariotta, Jack Davies, Bj{\o}rn Dundas and Gijs Heuts for helpful conversations related to the present work. Moreover, the author would like to thank Steffen Sagave for helpful discussions and comments on a draft version of this paper. Finally, the author is grateful to an anonymous referee for helpful comments on a previous version of this paper. This research was partially supported by the NWO-grant 613.009.121. 

\section{Commutative ${\cal J}$-space monoids}\label{commjspace} 

\subsection{Conventions on symmetric spectra} We denote by ${\rm Sp}^{\Sigma} = ({\rm Sp}^{\Sigma}, \wedge, {\Bbb S})$ the symmetric monoidal category of \emph{symmetric spectra} in simplicial sets as introduced in \cite{HSS00}. When referenced as a model category, we always reference the \emph{positive stable model structure} on ${\rm Sp}^{\Sigma}$, from which the category of commutative symmetric ring spectra ${\cal C}{\rm Sp}^{\Sigma}$ inherits a proper simplicial model structure \cite{MMSS01}. It is necessary to assume that a given commutative symmetric ring spectrum is positive fibrant for many of our constructions, and we will do so without further comment. 

\subsection{${\cal J}$-spaces} We give a rough recollection of the terminology involved in the theory of topological logarithmic structures based on commutative ${\cal J}$-space monoids as developed in \cites{Sag14, RSS15, RSS18}. A more detailed recollection is provided in \cite[Section 2]{RSS15}, while we refer to \cite[Section 4]{SS12} for further details. 

Denote by ${\cal J}$ the category obtained via Quillen's localization construction on the category of finite sets ${\bf n} = \{1, \dots, n\}$ and bijections. By convention, ${\bf 0}$ denotes the empty set. The category ${\cal J}$ has pairs $({\bf m}_1, {\bf m}_2)$ as objects. The category ${\cal S}^{\cal J}$ of \emph{${\cal J}$-spaces} is the category of functors from ${\cal J}$ to the category ${\cal S}$ of simplicial sets. This category has a symmetric monoidal product $\boxtimes$ defined as the left Kan extension along the ordered concatenation $- \sqcup - \colon {\cal J} \times {\cal J} \to {\cal J}$, i.e. \[(X \boxtimes Y)({\bf n_1, n_2}) = {\rm colim}_{(\bf k_1, k_2) \sqcup {(\bf l_1, l_2)} \to {(\bf n_1, n_2)}} X({\bf k_1, k_2}) \times Y({\bf l_1, l_2}).\] Here the colimit is taken in the slice category $\sqcup \downarrow {\cal J}$, and the monoidal unit $U^{\cal J}$ is the ${\cal J}$-space ${\cal J}(({\bf 0, 0}), -)$. The category ${\cal C}{\cal S}^{\cal J}$ of \emph{commutative ${\cal J}$-space monoids} is the category of commutative monoids with respect to this symmetric monoidal structure. 

By \cite[Proposition 4.10]{SS12}, the category of commutative ${\cal J}$-space monoids admits a \emph{positive ${\cal J}$-model structure} in which a map $M \to N$ is a weak equivalence if and only if it induces a weak equivalence $M_{h{\cal J}} \to N_{h{\cal J}}$ of simplicial sets on Bousfield--Kan homotopy colimits over ${\cal J}$. We refer to the weak equivalences of this model structure as \emph{${\cal J}$-equivalences}. This model structure is cofibrantly generated, proper and simplicial. With respect to this model structure there is a Quillen adjunction \[{\Bbb S}^{\cal J}[-] \colon {\cal C}{\cal S}^{\cal J} \leftrightarrows {\cal C}{\rm Sp}^{\Sigma} \colon \Omega^{\cal J}(-)\] relating the category of commutative ${\cal J}$-space monoids with that of commutative symmetric ring spectra. If $A$ is a commutative symmetric ring spectrum, we think of the commutative ${\cal J}$-space monoid $\Omega^{\cal J}(A)$ as the underlying graded monoid of~$A$. 

To every commutative ${\cal J}$-space monoid $M$ one can associate a graded discrete monoid $\pi_{0, *}(M)$ such that $\pi_{0, *}(\Omega^{\cal J}(A)) \cong (\pi_*(A), \cdot)$, the underlying graded multiplicative monoid of $\pi_*(A)$. There is a subobject ${\rm GL}_1^{\cal J}(A)$ of $\Omega^{\cal J}(A)$; the \emph{graded units} of $A$. As opposed to the ${\Bbb E}_{\infty}$-space of units ${\rm GL}_1(A)$ or its commutative ${\cal I}$-space model ${\rm GL}_1^{\cal I}(A)$, this captures units outside of $\pi_0(A)$; in particular, the associated map $\pi_{0, *}({\rm GL}_1^{\cal J}(A)) \to \pi_{0, *}(\Omega^{\cal J}(A))$ of graded discrete commutative monoids is the inclusion of the units in $(\pi_*(A), \cdot)$ \cite[Proposition 4.26]{SS12}.

We will need the following statement on numerous occasions:

\begin{lemma}\cite[Lemma 2.11]{Sag14}\label{monstr} Let $K$ and $L$ be commutative ${\cal J}$-space monoids, and assume that at least one of the two is cofibrant. Then the monoidal structure map \[K_{h{\cal J}} \times L_{h{\cal J}} \xrightarrow{\simeq} (K \boxtimes L)_{h{\cal J}}\] is a weak equivalence of simplicial sets.
\end{lemma}

In \emph{loc.\ cit.}\ the statement is proved under the weaker cofibrancy hypothesis of \emph{flatness}, however this will make no difference for us.

\subsection{The group completion model structure and repletion} We now discuss group completion of commutative ${\cal J}$-space monoids. A commutative ${\cal J}$-space monoid $N$ is \emph{grouplike} if $\pi_0(N_{h{\cal J}})$ is a group. By \cite[Theorem 5.5]{Sag16}, the category of commutative ${\cal J}$-space monoids admits a \emph{group completion model structure} in which $N \to M$ is a weak equivalence if and only if $B(N_{h{\cal J}}) \to B(M_{h{\cal J}})$ is a weak equivalence. The cofibrations coincide with those of the positive ${\cal J}$-model structure, while the fibrant objects are the positive fibrant $M$ which are grouplike. This model structure arises as a left Bousfield localization of the positive ${\cal J}$-model structure; in particular, it is left proper. 

\begin{definition} Let $N \to M$ be a morphism of commutative ${\cal J}$-space monoids.

\begin{enumerate}\item The \emph{group completion} of a $N$ is defined by a functorial factorization \[\begin{tikzcd}N \ar[tail]{r}{\simeq} & N^{\rm gp} \ar[two heads]{r} & *\end{tikzcd}\] in the group completion model structure. \item The \emph{repletion} $N^{\rm rep}$ of $N$ relative to $M$ is defined by a functorial factorization \[\begin{tikzcd}N \ar[tail]{r}{\simeq} & N^{\rm rep} \ar[two heads]{r} & M\end{tikzcd}\] in the group completion model structure (so that $N^{\rm rep} = N^{\rm gp}$ in the case where $M$ is terminal). \item The morphism $N \to M$ \emph{virtually surjective} if the induced map $\pi_0(N^{\rm gp}_{h{\cal J}}) \to \pi_0(M^{\rm gp}_{h{\cal J}})$ is a surjection of abelian groups.\end{enumerate}
\end{definition}

\begin{lemma}\label{repletepullback} Let $N \to M$ be a virtually surjective morphism of commutative ${\cal J}$-space monoids. There is a homotopy cartesian square of the form \[\begin{tikzcd}[row sep = small]N^{\rm rep} \ar{r} \ar{d} & (N^{\rm gp})^{\rm f} \ar{d} \\ M \ar{r} & M^{\rm gp} \end{tikzcd}\] in the positive ${\cal J}$-model structure, where $(N^{\rm gp})^{\rm f}$ participates in a factorization \[\begin{tikzcd}N^{\rm gp} \ar[tail]{r}{\simeq} & (N^{\rm gp})^{\rm f} \ar[two heads]{r} & M^{\rm gp}\end{tikzcd}\] in this model structure. In particular, the square \[\begin{tikzcd}[row sep = small]N^{\rm rep} \ar{r} \ar{d} & (N^{\rm rep})^{\rm gp} \ar{d} \\ M \ar{r} & M^{\rm gp}\end{tikzcd}\] is homotopy cartesian.
\end{lemma}

\begin{remark}\label{abuse} In situations where we apply the latter formulation of Lemma \ref{repletepullback}, we will abuse notation slightly and simply write $N^{\rm gp}$ in place of $(N^{\rm rep})^{\rm gp}$. 
\end{remark}

\begin{proof}[Proof of Lemma \ref{repletepullback}] The existence of the first homotopy cartesian square is the content of \cite[Lemma 3.17]{RSS15}. To see that the second square is homotopy cartesian, we use the lifting properties of the group completion model structure to obtain a ${\cal J}$-equivalence $(N^{\rm rep})^{\rm gp} \xrightarrow{\simeq} (N^{\rm gp})^{\rm f}$ over $M^{\rm gp}$. Since the positive ${\cal J}$-model structure is right proper, \cite[Proposition 13.3.4]{Hir03} applies to compare the two squares.
\end{proof}

\subsection{Group completion commutes with homotopy pushouts} Suppose that we are given a diagram $M \xleftarrow{} P \xrightarrow{} N$ of commutative ${\cal J}$-space monoids with $P$ cofibrant and $P \to M$ a cofibration. We may assume without loss of generality that $P^{\rm gp} \to M^{\rm gp}$ is a cofibration as well, so that the induced morphism \[M \boxtimes_P N \xrightarrow{} M^{\rm gp} \boxtimes_{P^{\rm gp}} N^{\rm gp}\] is a weak equivalence in the group completion model structure. The universal property of the pushout provides a natural map \begin{equation}\label{gpweq}M^{\rm gp} \boxtimes_{P^{\rm gp}} N^{\rm gp} \xrightarrow{} (M \boxtimes_P N)^{\rm gp}\end{equation} which is a weak equivalence in the group completion model structure.

\begin{lemma}\label{gpcomppush} In the context described above, the morphism (\ref{gpweq}) is a ${\cal J}$-equivalence. 
\end{lemma}

\begin{proof}  Since $M^{\rm gp}$ is cofibrant, the monoidal structure map \[M^{\rm gp}_{h{\cal J}} \times N^{\rm gp}_{h{\cal J}} \xrightarrow{\simeq} (M^{\rm gp} \boxtimes N^{\rm gp})_{h{\cal J}}\] is a weak equivalence of simplicial sets by Lemma \ref{monstr}. Hence the map \[\pi_0((M^{\rm gp} \boxtimes N^{\rm gp})_{h{\cal J}}) \xrightarrow{} \pi_0((M^{\rm gp} \boxtimes_{P^{\rm gp}} N^{\rm gp})_{h{\cal J}})\] is a surjection. The domain of this map being a group implies that its codomain is a group as well. By definition, this means that $M^{\rm gp} \boxtimes_{P^{\rm gp}} N^{\rm gp}$ is grouplike, which concludes the proof since (\ref{gpweq}) is now a weak equivalence in the group completion model structure between grouplike commutative ${\cal J}$-space monoids.
\end{proof}

\subsection{The repletion of augmented commutative ${\cal J}$-space monoids} Let $P$ be a cofibrant commutative ${\cal J}$-space monoid and let $P \to M \to P$ be a cofibrant commutative ${\cal J}$-space monoid over and under $P$, that is, the morphism $P \to M$ is a cofibration of commutative ${\cal J}$-space monoids and the composite is the identity on $P$. This clearly implies that the augmentation morphism $M \to P$ is virtually surjective, and as such we may describe the repletion $M^{\rm rep}$ of this morphism as the (homotopy) pullback of  $P \xrightarrow{} P^{\rm gp} \xleftarrow{} (M^{\rm gp})^{\rm f}$ as in Lemma \ref{repletepullback}. Recall that $U^{\cal J}$ is the initial object in the category of commutative ${\cal J}$-space monoids; its homotopy colimit $U^{\cal J}_{h{\cal J}}$ over ${\cal J}$ is contractible.

\begin{definition}\label{w} Define $W(M)$ as the (homotopy) pullback of the diagram \[U^{\cal J} \xrightarrow{} P^{\rm gp} \xleftarrow{} (M^{\rm gp})^{\rm f}\] of commutative ${\cal J}$-space monoids. 
\end{definition}

The universal property of the coproduct induces a natural map \begin{equation}\label{splitoffreplete}P \boxtimes W(M) \to P \times_{P^{\rm gp}} (M^{\rm gp})^{\rm f}\end{equation} over and under $P$. In the special case where $M = B^{\rm cy}(P)$ is the cyclic bar construction on $P$, the content of the following lemma appears implicitly in \cite[Proof of Proposition 3.1]{RSS18}. Our definition of $W(M)$ is analogous to (but different from) a similar construction in \emph{loc.\ cit}. In the context of discrete commutative monoids, a version of this result appears in \cite[Lemma 3.11]{Rog09}.

\begin{lemma}\label{repleterewrite} The map (\ref{splitoffreplete}) is a ${\cal J}$-equivalence.
\end{lemma}

\begin{proof} Let $B^{\boxtimes}(-, -, -)$ denote the two-sided bar construction in ${\cal J}$-spaces. It suffices to show that the square  \[\begin{tikzcd}[row sep = small]B^{\boxtimes}(P, U^{\cal J}, W(M)) \ar{r} \ar{d} & B^{\boxtimes}(P^{\rm gp}, P^{\rm gp}, (M^{\rm gp})^{\rm f}) \ar{d} \\ B^{\boxtimes}(P, U^{\cal J}, U^{\cal J}) \ar{r} & B^{\boxtimes}(P^{\rm gp}, P^{\rm gp}, P^{\rm gp}) \end{tikzcd}\] is homotopy cartesian in the positive ${\cal J}$-model structure. We apply \cite[Corollary 11.4]{SS12}, which states that $(-)_{h{\cal J}}$ detects and preserves homotopy cartesian squares.  Combining this with our cofibrancy hypotheses and Lemma \ref{monstr}, we reduce to proving that the square \begin{equation}\label{bf2}\begin{tikzcd}[row sep = small]B^{\times}(P_{h{\cal J}}, U^{\cal J}_{h{\cal J}}, W(M)_{h{\cal J}}) \ar{r} \ar{d} & B^{\times}(P^{\rm gp}_{h{\cal J}}, P^{\rm gp}_{h{\cal J}}, (M^{\rm gp})^{\rm f}_{h{\cal J}}) \ar{d} \\ B^{\times}(P_{h{\cal J}}, U^{\cal J}_{h{\cal J}}, U^{\cal J}_{h{\cal J}}) \ar{r} & B^{\times}(P^{\rm gp}_{h{\cal J}}, P^{\rm gp}_{h{\cal J}}, P^{\rm gp}_{h{\cal J}}) \end{tikzcd}\end{equation} is a homotopy cartesian square of simplicial sets.

For this we apply the Bousfield--Friedlander theorem \cite[Theorem B.4]{BF78}. The square arises as the realization of a square of bisimplicial sets, and it is clear that this square is pointwise homotopy cartesian. The bisimplicial sets \[B^{\times}_{\bullet}(P^{\rm gp}_{h{\cal J}}, P^{\rm gp}_{h{\cal J}}, (M^{\rm gp})^{\rm f}_{h{\cal J}}) \text{      and      }   B^{\times}_{\bullet}(P^{\rm gp}_{h{\cal J}}, P^{\rm gp}_{h{\cal J}}, P^{\rm gp}_{h{\cal J}}) \] satisfy the $\pi_*$-Kan condition as the simplicial commutative monoids involved are grouplike. Moreover, the map between them induces a Kan fibration on vertical path components, as the map on vertical path components is a surjection of simplicial abelian groups. Hence the Bousfield--Friedlander theorem is applicable to infer that the square (\ref{bf2}) is homotopy cartesian, which concludes the proof. 
\end{proof}

\section{The cyclic bar construction and simplicial tensors}\label{cycbar} 

We review the cyclic and replete bar constructions, and we define topological Hochschild homology as an instance of the cyclic bar construction in the category of symmetric spectra. Our approach to this content largely follows that of \cite[Section 3]{RSS15}, with the appropriate modifications and additions necessary to pass from the absolute context in \emph{loc.\ cit.}\ to the relative one to be discussed here.  

\subsection{The relative cyclic bar construction} Throughout this section we denote by ${\cal M} = ({\cal M}, \boxtimes, U)$ a cocomplete symmetric monoidal category. The category ${\cal C}{\cal M}$ of commutative monoids in ${\cal M}$ is also cocomplete in this case. The coproduct in ${\cal C}{\cal M}$ is given by the symmetric monoidal product in ${\cal M}$. Given a diagram $M \xleftarrow{} P \xrightarrow{} N$ in ${\cal C}{\cal M}$, we write $M \boxtimes_P N$ for its pushout. This is the coproduct of $M$ and $N$ in the category ${\cal C}{\cal M}_{P/}$ of commutative monoids over $P$.

\begin{definition} Let $P \to M$ be a morphism in ${\cal C}{\cal M}$. The \emph{cyclic bar construction} $B^{\rm cy}_P(M)_{\bullet}$ is the following simplicial object in ${\cal C}{\cal M}$: the $q$-simplices are given by the $(1 + q)$-fold coproduct \[M \boxtimes_P M \boxtimes_P \cdots \boxtimes_P M\] of $M$ in ${\cal C}{\cal M}_{P/}$. The simplicial structure maps are informally given as follows (see e.g.\ \cite[Section 3]{RSS15} for the precise definitions): the $i$th face map uses the multiplication map $M \boxtimes_P M \to M$ on the $i$th and $(i + 1)$st factor for as long as it makes sense to do so. The last face map multiplies the last and first factor together by precomposing with one of the symmetry isomorphisms. The $j$th face map inserts the unit between the $(j - 1)$st and $j$th factor. We omit $P$ from the notation when it equals the monoidal unit. 
\end{definition}

Notice that the isomorphisms (\emph{resp.} iterated multiplication maps) \[M \xrightarrow{\cong} M \boxtimes_P P \boxtimes_P \cdots \boxtimes_P P \qquad \text{(\emph{resp. }} M \boxtimes_P M \boxtimes_P \cdots \boxtimes_P M \xrightarrow{} M)\] exhibit $B^{\rm cy}_P(M)_{\bullet}$ as a simplicial object in the pointed catgeory ${\cal C}{\cal M}_{M // M}$ of commutative monoids in ${\cal M}$ over and under $M$. 

\begin{definition}\label{tensors} Let $P \to M$ be a morphism in ${\cal C}{\cal M}$ and let $L$ be an object in ${\cal C}{\cal M}_{K // K}$ for a given commutative monoid $K$ in ${\cal M}$. 
\begin{enumerate}
\item Let $X$ be a finite simplicial set. The \emph{tensor} $X_{\bullet} \otimes_P M$ is defined as the simplicial object $[q] \mapsto M^{\boxtimes_P X_q}$ in ${\cal C}{\cal M}_{P/}$. The $q$-simplices are the $|X_q|$-fold coproduct of $M$ in ${\cal C}{\cal M}_{P/}$, and the simplicial structure arises via that of $X_{\bullet}$, and the multiplication and unit maps of $M$.
\item Let $X = (X, *)$ be a finite \emph{pointed} simplicial set. The \emph{pointed tensor} $X \odot_K L$ is defined by the pushout of the diagram $K \xleftarrow{} L \xrightarrow{} X_{\bullet} \otimes_K L$ in the category of simplicial objects in ${\cal C}{\cal M}_{K // K}$. The map $L \to X_{\bullet} \otimes_K L$ is induced by the basepoint in $X_{\bullet}$. 
\end{enumerate}
Both of the above constructions may be extended to allow for arbitrary simplicial sets by realizing a given set as a colimit of its finite subsets.
\end{definition}

We shall use the geometric realizations of the above constructions in the simplicial model categories ${\cal C}{\cal S}^{\cal J}$ and ${\cal C}{\rm Sp}^{\Sigma}$, in which case we omit the bullets from the notation. In these cases, the unpointed tensor participates in the simplicial structure of the respective positive model structures. In particular, the functors $X \otimes -$ are left Quillen for any simplicial set $X$. 

The following result summarizes the properties that we shall use about the relative cyclic bar construction and its relation to simplicial tensors. We model the simplicial circle by $S^1_{\bullet} := \Delta[1]/\partial\Delta[1]$, which is pointed at its unique $0$-simplex. 

\begin{proposition}\label{prop:cycbar} Let $P \to M$ be a morphism of commutative monoids in ${\cal M}$. There are natural isomorphisms \[P \boxtimes_{B^{\rm cy}(P)_{\bullet}} B^{\rm cy}(M)_{\bullet} \cong B^{\rm cy}_P(M)_{\bullet} \cong S^1_{\bullet} \otimes_P M \cong S^1_{\bullet} \odot_M (M \boxtimes_P M)\] over and under $M$. 
\end{proposition}

\begin{proof} The first isomorphism follows from the definition of the cyclic bar construction. For example, for the $1$-simplices, this is the isomorphism \[P \boxtimes_{P \boxtimes_U P} (M \boxtimes_U M) \cong (P \boxtimes_P P) \boxtimes_{P \boxtimes_U P} (M \boxtimes_U M) \cong M \boxtimes_P M,\] where the last isomorphism follows from commuting colimits.  For the second isomorphism, the argument given in \cite[Lemma 3.3]{RSS15} applies also for this relative construction, while the third similarly follows from the definition of the pointed tensor. 
\end{proof}

\subsection{Group completion of the cyclic bar construction} We will use the fact that group completion interacts well with the cyclic bar construction of commutative ${\cal J}$-space monoids.

\begin{lemma}\label{chainovergp} Let $P \to M$ be a cofibration of cofibrant commutative ${\cal J}$-space monoids. There is a chain of ${\cal J}$-equivalences under $B^{\rm cy}_P(M)$ and over $M^{\rm gp}$ relating $B^{\rm cy}_{P^{\rm gp}}(M^{\rm gp})$ and $B^{\rm cy}_P(M)^{\rm gp}$. 
\end{lemma}

\begin{proof} In \cite[Proof of Lemma 3.19]{RSS15}, it is proven that $S^1 \otimes P^{\rm gp} \cong B^{\rm cy}(P^{\rm gp})$ is grouplike since $P^{\rm gp}$ is. Hence this is also the case for $B^{\rm cy}_{P^{\rm gp}}(M^{\rm gp})$, as it is isomorphic to $P^{\rm gp} \boxtimes_{B^{\rm cy}(P^{\rm gp})} B^{\rm cy}(M^{\rm gp})$ by Proposition \ref{prop:cycbar}. Since $P \to M$ is assumed to be a cofibration, there is no loss of generality in assuming that $P^{\rm gp} \to M^{\rm gp}$ is a cofibration as well. In particular, $B^{\rm cy}(P^{\rm gp}) \to B^{\rm cy}(M^{\rm gp})$ is a cofibration in this case, and so the morphism \[B^{\rm cy}_P(M) \cong P \boxtimes_{B^{\rm cy}(P)} B^{\rm cy}(M) \xrightarrow{} P^{\rm gp} \boxtimes_{B^{\rm cy}(P^{\rm gp})} B^{\rm cy}(M^{\rm gp}) \cong B^{\rm cy}_{P^{\rm gp}}(M^{\rm gp})\] is a weak equivalence in the group completion model structure. The lifting properties of the group completion model structure provides the dashed arrow in the diagram \[\begin{tikzcd}[row sep = small]B^{\rm cy}_P(M) \ar[tail]{d}{\simeq_{\rm gp}} \ar{r}{\simeq_{\rm gp}} & B^{\rm cy}_{P^{\rm gp}}(M^{\rm gp}) \ar[tail]{r}{\simeq_{\rm gp}} & (B^{\rm cy}_{P^{\rm gp}}(M^{\rm gp}))'\ar[two heads]{d} \\ B^{\rm cy}_P(M)^{\rm gp} \ar{rr} \ar[dashed]{urr} & & M^{\rm gp}, \end{tikzcd}\] which is a weak equivalence in this model structure. Here $(B^{\rm cy}_{P^{\rm gp}}(M^{\rm gp}))'$ appears in the indicated factorization of the augmentation $B^{\rm cy}_{P^{\rm gp}}(M^{\rm gp}) \to M^{\rm gp}$. This concludes the proof, as weak equivalences in the group completion model structure between grouplike commutative ${\cal J}$-space monoids are ${\cal J}$-equivalences.
\end{proof}

\subsection{Topological Hochschild homology} We use the cyclic bar construction to model topological Hochschild homology:

\begin{definition} Let $R \to A$ be a cofibration of cofibrant commutative symmetric ring spectra. The \emph{topological Hochschild homology} ${\rm THH}^R(A)$ is the commutative symmetric ring spectrum $B^{\rm cy}_R(A)$ given by the geometric realization of the cyclic bar construction $A$ relative to $R$ in symmetric spectra.
\end{definition}

When $R = {\Bbb S}$ is the sphere spectrum, we shall simply write ${\rm THH}(A)$ in place of ${\rm THH}^{\Bbb S}(A)$. We remark that it follows from strong symmetric monoidality of the functor ${\Bbb S}^{\cal J}[-]$ that ${\rm THH}^{{\Bbb S}^{\cal J}[P]}({\Bbb S}^{\cal J}[M]) \cong {\Bbb S}^{\cal J}[B^{\rm cy}_P(M)]$. We deduce the following from Proposition \ref{prop:cycbar}:

\begin{proposition}\label{thhsuspension} Let $R \to A$ be a cofibration of cofibrant commutative symmetric ring spectra. There is a chain of natural isomorphisms \[R \wedge_{{\rm THH}(R)} {\rm THH}(A) \cong {\rm THH}^R(A) \cong S^1 \otimes_R A \cong S^1 \odot_A (A \wedge_R A)\] of augmented commutative $A$-algebras. \qed
\end{proposition}

Here $S^1 \odot_A -$ is the pointed simplicial tensor with $S^1$ in the pointed category ${\cal C}{\rm Sp}^{\Sigma}_{A // A}$ of commutative augmented $A$-algebras. This is a model for the suspension functor in ${\cal C}{\rm Sp}^{\Sigma}_{A // A}$. As we explain in Section \ref{lastsection}, this structure is essential for the proof of Theorem \ref{mainthm}.

\subsection{The relative replete bar construction} We now discuss the repletion of the augmented commutative ${\cal J}$-space monoid \[M \to B^{\rm cy}_P(M) \to M.\] Due to its central role, we single out this case as a seperate definition:

\begin{definition} Let $P \to M$ be a cofibration of cofibrant commutative ${\cal J}$-space monoids. The \emph{replete bar construction} $B^{\rm cy}_P(M)^{\rm rep}$ is the repletion of the augmentation map $B^{\rm cy}_P(M) \to M$. 
\end{definition}

This is equivalent (but not equal) to the replete bar construction discussed in \cite{RSS15}, see Remark \ref{logthhcircle}.

The isomorphism $B^{\rm cy}(P) \cong S^1 \otimes P$ of Proposition \ref{prop:cycbar} gives that the map $B^{\rm cy}(P) \to B^{\rm cy}(M)$ is a cofibration if $P \to M$ is. Therefore there is no loss of generality in assuming that $B^{\rm cy}(P)^{\rm rep} \to B^{\rm cy}(M)^{\rm rep}$ is a cofibration as well. We now wish to prove an analogue of the isomorphism $P \boxtimes_{B^{\rm cy}(P)} B^{\rm cy}(M) \cong B_P^{\rm cy}(M)$ of Proposition \ref{prop:cycbar} for the replete bar construction. The universal property of the pushout provides a natural map $P \boxtimes_{B^{\rm cy}(P)^{\rm rep}} B^{\rm cy}(M)^{\rm rep} \xrightarrow{} (P 
\boxtimes_{B^{\rm cy}(P)} B^{\rm cy}(M))^{\rm rep}$ of commutative ${\cal J}$-space monoids.

\begin{lemma}\label{repbarjuggling} Let $P \to M$ be as above. There are ${\cal J}$-equivalences \[P \boxtimes_{B^{\rm cy}(P)^{\rm rep}} B^{\rm cy}(M)^{\rm rep} \xrightarrow{\simeq} (P \boxtimes_{B^{\rm cy}(P)} B^{\rm cy}(M))^{\rm rep} \xrightarrow{\cong} B^{\rm cy}_P(M)^{\rm rep}\] under $B^{\rm cy}_P(M)$ and over $M$.  
\end{lemma}

\begin{proof} The second map in the composite is an isomorphism since it arises from applying $(-)^{\rm rep}$ to the isomorphism $P \boxtimes_{B^{\rm cy}(P)} B^{\rm cy}(M) \cong B^{\rm cy}_P(M)$ of Proposition~\ref{prop:cycbar}. Applying the functor $(-)^{\rm gp}$ to this isomorphism and Lemma \ref{gpcomppush}, we see that it suffices by Lemma \ref{repletepullback} to argue that \begin{equation}\label{eqrepl2}\begin{tikzcd}[row sep = small]P \boxtimes_{B^{\rm cy}(P)^{\rm rep}} B^{\rm cy}(M)^{\rm rep} \ar{r} \ar{d} & P^{\rm gp} \boxtimes_{B^{\rm cy}(P)^{\rm gp}} B^{\rm cy}(M)^{\rm gp} \ar{d} \\ M \ar{r} & M^{\rm gp} \end{tikzcd}\end{equation} is homotopy cartesian. To see this, we model homotopy pushouts of commutative ${\cal J}$-space monoids by the two-sided bar construction $B^{\boxtimes}(-, -, -)$, and notice that the square in question may be rewritten as \[\begin{tikzcd}[row sep = small]B^{\boxtimes}(P, B^{\rm cy}(P)^{\rm rep}, B^{\rm cy}(M)^{\rm rep}) \ar{r} \ar{d} & B^{\boxtimes}(P^{\rm gp}, B^{\rm cy}(P)^{\rm gp}, B^{\rm cy}(M)^{\rm gp}) \ar{d} \\ B^{\boxtimes}(P, P, M) \ar{r} & B^{\boxtimes}(P^{\rm gp}, P^{\rm gp}, M^{\rm gp}).\end{tikzcd}\] Since the squares \[\begin{tikzcd}[row sep = small]B^{\rm cy}(P)^{\rm rep} \ar{r} \ar{d} & B^{\rm cy}(P)^{\rm gp} \ar{d} \\ P \ar{r} & P^{\rm gp}\end{tikzcd} \text{ and } \begin{tikzcd}[row sep = small]B^{\rm cy}(M)^{\rm rep} \ar{r} \ar{d} & B^{\rm cy}(M)^{\rm gp} \ar{d} \\ M \ar{r} & M^{\rm gp}\end{tikzcd}\] are homotopy cartesian, we may argue with the Bousfield--Friedlander theorem as in the proof of Lemma \ref{repleterewrite} to conclude that (\ref{eqrepl2}) is homotopy cartesian. \end{proof}

We now provide an analogue of the isomorphism $B^{\rm cy}_P(M) \cong S^1 \odot_M (M \boxtimes_P M)$ of Proposition \ref{prop:cycbar}, which is an essential ingredient in the description of logarithmic ${\rm THH}$ as a suspension in a category of augmented algebras given in Proposition \ref{logthhsuspension}. 

\begin{proposition}\label{prop:repbarsuspension} There is a chain of ${\cal J}$-equivalences \[S^1 \odot_M (M \boxtimes_P M)^{\rm rep} \simeq (S^1 \odot_M (M \boxtimes_P M))^{\rm rep} \xrightarrow{\cong} B^{\rm cy}_P(M)^{\rm rep}\] under $B^{\rm cy}(M)$ and over $M$, where all repletions are taken with respect to the natural augmentations to $M$. 
\end{proposition}

\begin{proof} The second map is an isomorphism since it arises from applying $(-)^{\rm rep}$ to the isomorphism of Proposition \ref{prop:cycbar}. By Lemma \ref{repletepullback} and Remark \ref{abuse}, we know that there is a homotopy cartesian square \begin{equation}\label{newproofsquare1}\begin{tikzcd}[row sep = small](S^1 \odot_M (M \boxtimes_P M))^{\rm rep} \ar{r} \ar{d} & (S^1 \odot_M (M \boxtimes_P M))^{\rm gp} \ar{d} \\ M \ar{r} & M^{\rm gp}\end{tikzcd}\end{equation} with respect to the positive ${\cal J}$-model structure.  We observe that there are ${\cal J}$-equivalences \begin{align*}(S^1 \odot_M (M \boxtimes_P M))^{\rm gp} & \cong (B^{\rm cy}_P(M))^{\rm gp}  \simeq B^{\rm cy}_{P^{\rm gp}}(M^{\rm gp})  \\ & \xleftarrow{\cong} S^1 \odot_{M^{\rm gp}}  (M^{\rm gp} \boxtimes_{P^{\rm gp}} M^{\rm gp}) \xrightarrow{\simeq} S^1 \odot_{M^{\rm gp}} (M \boxtimes_P M)^{\rm gp}\end{align*} over $M^{\rm gp}$. Here the first isomorphism arises from applying $(-)^{\rm gp}$ to the isomorphism of Proposition \ref{prop:cycbar}. The chain of $\cal J$-equivalences $(B^{\rm cy}_P(M))^{\rm gp} \simeq B^{\rm cy}_{P^{\rm gp}}(M^{\rm gp})$ refers to the chain constructed in Lemma \ref{chainovergp}, the following isomorphism is another application of Proposition \ref{prop:cycbar}, while the last ${\cal J}$-equivalence comes from applying the left Quillen functor $S^1 \odot_{M^{\rm gp}} -$ to the ${\cal J}$-equivalence of cofibrant objects $M^{\rm gp} \boxtimes_{P^{\rm gp}} M^{\rm gp} \xrightarrow{\simeq} (M \boxtimes_P M)^{\rm gp}$ from  Lemma \ref{gpcomppush}. 

Following the same strategy as in the proof of Lemma \ref{repleterewrite}, we observe that the square \begin{equation}\label{newproofsquare2}\begin{tikzcd}[row sep = small] B^\boxtimes(M, (M \boxtimes_P M)^{\rm rep}, M) \ar{r} \ar{d} & B^\boxtimes(M^{\rm gp}, (M \boxtimes_P M)^{\rm gp}, M^{\rm gp}) \ar{d} \\ B^{\boxtimes}(M, M, M) \ar{r} & B^\boxtimes(M^{\rm gp}, M^{\rm gp}, M^{\rm gp})\end{tikzcd}\end{equation} is homotopy cartesian by the Bousfield--Friedlander theorem, Lemma \ref{repletepullback} and Remark \ref{abuse}. Since the two-sided bar construction $B^{\boxtimes}(M^{\rm gp}, (M \boxtimes_P M)^{\rm gp}, M^{\rm gp})$ is a model for the suspension $S^1 \odot_{M^{\rm gp}} (M \boxtimes_P M)^{\rm gp}$, the result follows by comparing the homotopy cartesian squares (\ref{newproofsquare1}) and (\ref{newproofsquare2}). 
\end{proof}

\section{Logarithmic ring spectra}\label{logringspectra} We introduce the necessary background material on logarithmic ring spectra. Our main references for this section are \cite[Section 4]{RSS15} and \cite[Section 4]{Sag14}. 

\begin{definition} A \emph{pre-logarithmic ring spectrum} $(A, M) = (A, M, \alpha)$ consists of a commutative symmetric ring spectrum $A$, a commutative ${\cal J}$-space monoid $M$ and a morphism of commutative ${\cal J}$-space monoids $\alpha \colon M \to \Omega^{\cal J}(A)$. It is a \emph{logarithmic ring spectrum} if the map $\alpha^{-1}{\rm GL}_1^{\cal J}(A) \to {\rm GL}_1^{\cal J}(A)$ in the (homotopy) pullback square \begin{equation}\label{logificationpullback}\begin{tikzcd}[row sep = small]\alpha^{-1}{\rm GL}_1^{\cal J}(A) \ar{r} \ar{d} & {\rm GL}_1^{\cal J}(A) \ar{d} \\ M \ar{r}{\alpha} & \Omega^{\cal J}(A)\end{tikzcd}\end{equation} is a ${\cal J}$-equivalence. A morphism $(f, f^{\flat}) \colon (R, P) \to (A, M)$ of pre-log ring spectra consists of a map of commutative symmetric ring spectra $f \colon R \to A$ and a map of commutative ${\cal J}$-space monoids $f^{\flat} \colon P \to M$ such that $\Omega^{\cal J}(f) \circ \beta = \alpha \circ f^{\flat}$.
\end{definition}

Using basic model category techniques, we find that the category ${\cal P} = {\rm PreLog}$ of pre-logarithmic ring spectra admits a \emph{projective} model structure in which a map $(f, f^{\flat})$ is a weak equivalence or fibration if and only if $f$ and $f^{\flat}$ is a weak equivalence or fibration. We choose this model structure for consistency with \cite{RSS15, RSS18}; as we explain in Remark \ref{logthhcircle}, the nature of our constructions make our arguments go through for the injective model structure as well. This may be advantageous if one wants to employ spectral variants of the \emph{log} model structure constructed in \cite[Section 3]{SSV16}, but we have not made use of this material here. 

\begin{example}\label{ex:prelogstr} We provide a series of natural examples of pre-log ring spectra and maps relating them:

\begin{enumerate}

\item If $A$ is a commutative symmetric ring spectrum, the inclusion ${\rm GL}_1^{\cal J}(A) \subset \Omega^{\cal J}(A)$ gives rise to the \emph{trivial log structure} on $A$, and $(A, {\rm GL}_1^{\cal J}(A))$ is the \emph{trivial log ring spectrum}. 

\item If $(A, M)$ is a pre-log ring spectrum, there is a map \[(A, M) \to (A \wedge_{{\Bbb S}^{\cal J}[M]} {\Bbb S}^{\cal J}[M^{\rm gp}], {\rm GL}_1^{\cal J}(A \wedge_{{\Bbb S}^{\cal J}[M]} {\Bbb S}^{\cal J}[M^{\rm gp}]))\] from $(A, M)$ to its \emph{localization} $A[M^{-1}] := A \wedge_{{\Bbb S}^{\cal J}[M]} {\Bbb S}^{\cal J}[M^{\rm gp}]$ equipped with its trivial log structure. 

\item Let $A$ be a discrete valuation ring and let $\pi \in A$ be a uniformizer. Denote by $\langle x \rangle$ the free commutative monoid on one generator $x$, and define a map $\langle x \rangle \to (A, \cdot)$ by sending $x$ to $\pi$. Then $(A, \langle x \rangle)$ is a pre-log ring, and we obtain a pre-log ring spectrum by $(HA, F_{(\bf 0, 0)}^{\cal J}\langle x \rangle)$, where $H(-)$ denotes the Eilenberg--Mac\,Lane spectrum and $F_{(\bf 0, 0)}^{\cal J}(-)$ is left adjoint to the evaluation functor sending a commutative ${\cal J}$-space monoid $M$ to $M({\bf 0, 0})$. 

\item Let $A$ be a commutative symmetric ring spectrum and let $x \in \pi_{n_2 - n_1}(A)$ be a homotopy class, represented by a map $x \colon S^{n_2} \to A_{n_1}$. Then $x$ can be regarded as a point  \[x \colon * \to \Omega^{\cal J}(A)({\bf n_1, n_2}) := \Omega^{n_2}(A_{n_1}).\] The point $x$ is adjoint to a morphism ${\Bbb C}(x) \to \Omega^{\cal J}(A)$ of commutative ${\cal J}$-space monoids, where ${\Bbb C}(x)$ denotes the free commutative ${\cal J}$-space monoid on a point in degree $({\bf n_1, n_2})$. Its localization is the commutative symmetric ring spectrum $A[1/x]$ by \cite[Proposition 3.19]{Sag14}. 

\item Building on the previous example, we define the \emph{direct image pre-log structure} $D(x)$ associated with $x$ as follows: form a homotopy pullback square \[\begin{tikzcd}[row sep = small]D'(x) \ar{r} \ar{d} & {\Omega}^{\cal J}(A) \ar{d} \\ {\Bbb C}(x)^{\rm gp} \ar{r} & \Omega^{\cal J}(A[1/x])\end{tikzcd}\] and define $D(x)$ via a cofibrant replacement of $D'(x)$ relative to ${\Bbb C}(x)$; see \cite[Construction 4.2]{Sag14} for the precise construction. The localization of this pre-log ring spectrum is also $A[1/x]$ by \cite[Theorem 4.4]{Sag14}. This construction is independent of the choice of the representative $x$ and will be our preferred choice of pre-log structure associated to a homotopy class. 
\end{enumerate}
\end{example}

\subsection{The logification construction} We now recall a functorial procedure for passing from a pre-log ring spectrum $(A, M)$ to a log ring spectrum $(A, M^a)$:

\begin{construction} Let $(A, M) = (A, M, \alpha)$ be a pre-log ring spectrum. Form a factorization \[\begin{tikzcd}\alpha^{-1}{\rm GL}_1^{\cal J}(A) \ar[tail]{r} & G \ar[two heads]{r}{\simeq} & {\rm GL}_1^{\cal J}(A)\end{tikzcd}\] in the positive ${\cal J}$-model structure of the natural map $\alpha^{-1}{\rm GL}_1^{\cal J}(A) \to {\rm GL}_1^{\cal J}(A)$ (see (\ref{logificationpullback})), and consider the (homotopy) pushout \[\begin{tikzcd}[row sep = small]\alpha^{-1}{\rm GL}_1^{\cal J}(A) \ar{r} \ar{d} & G \ar{d} \\ M \ar{r} & M^a\end{tikzcd}\] of commutative ${\cal J}$-space monoids. The maps \[G \to {\rm GL}_1^{\cal J}(A) \to \Omega^{\cal J}(A) \text{ and } M \xrightarrow{\alpha} \Omega^{\cal J}(A)\] give rise to a map $\alpha^a \colon M^a \to \Omega^{\cal J}(A)$, and $(A, M^a, \alpha^a)$ is called the \emph{logification} of $(A, M, \alpha)$. By \cite[Lemma 3.12]{Sag14}, this is indeed a log structure.
\end{construction}

We discuss the effect of the logification construction for some of the pre-log ring spectra discussed in Example \ref{ex:prelogstr}. 

\begin{example}\label{ex:logification1} Consider the discrete pre-log ring $(A, \langle x \rangle)$ from Example \ref{ex:prelogstr}(3). Its logification (in discrete pre-log rings) is the log ring $(A, A \cap {\rm GL}_1(K))$, where $K$ denotes the fraction field of the discrete valuation ring $A$ and $A \cap {\rm GL}_1(K) = \langle \pi \rangle \times {\rm GL}_1(A)$ consists of all non-zero elements of $A$. As there is a pullback square \[\begin{tikzcd}[row sep = small]A \cap {\rm GL}_1(K) \ar{r} \ar{d} &  (A, \cdot) \ar{d} \\ {\rm GL}_1(K) \ar{r} & (K, \cdot),\end{tikzcd}\] this is an instance of a \emph{direct image log structure} on $A$ induced by the trivial log structure on $K$ and the localization map $A \to K$. This example is a special case of \cite[Example (2.5)]{Kat89}.
\end{example}

\begin{example}\label{ex:logification2} For the next examples, we fix the following setup: the commutative symmetric ring spectrum $A$ is connective and the map $x \colon S^{n_2} \to A_{n_1}$ represents a nontrivial homotopy class in  $\pi_{n_2 - n_1}(A)$ of even positive degree, and the localization map $A \to A[1/x]$ is a model for the connective cover map of $A[1/x]$. Examples of this kind include the real and complex connective $K$-theory spectra ${\rm ko}$ and ${\rm ku}$ and the Adams summand ${\ell}$.
\begin{enumerate}
\item Consider the pre-log ring spectrum $(A, {\Bbb C}(x))$ discussed in Example~\ref{ex:prelogstr}(4). By \cite[Lemma 4.9]{Sag14}, the associated log structure is weakly equivalent to $(A, {\Bbb C}(x) \boxtimes {\rm GL}_1^{\cal J}(A))$. 
\item Consider the pre-log ring spectrum $(A, D(x))$ discussed in Example~\ref{ex:prelogstr}(5). By \cite[Lemma 4.7]{Sag14}, the associated log structure is weakly equivalent to $(A, j_*{\rm GL}_1^{\cal J}(A[1/x]))$, the direct image log structure associated to the connective cover map $j \colon A \to A[1/x]$ and the trivial log structure on $A[1/x]$. 
\end{enumerate}
\end{example}

The two log ring spectra above are not weakly equivalent, and this displays an interesting distinction which is not visible for discrete pre-log rings, as Example \ref{ex:logification1} illustrates in the case of discrete valuation rings. We refer to \cite[Remark 4.8]{Sag14} for further comments in this direction.  

\subsection{Mapping spaces of pre-log ring spectra} Let $(A, M)$ and $(B, N)$ be pre-logarithmic ring spectra.

\begin{definition}\label{logmappingspace} The \emph{space of maps} ${\rm Map}_{\cal P}((A, M), (B, N))$ is the pullback of \[{\rm Map}_{{\cal C}{\cal S}^{\cal J}}(M, N) \xrightarrow{} {\rm Map}_{{\cal C}{\rm Sp}^{\Sigma}}({\Bbb S}^{\cal J}[M], B) \xleftarrow{} {\rm Map}_{{\cal C}{\rm Sp}^{\Sigma}}(A, B),\] where the morphisms are induced by the structure maps. This captures a well-defined homotopy type as soon as $(A, M)$ is cofibrant and $(B, N)$ is fibrant, as the structure map ${\Bbb S}^{\cal J}[M] \to A$ is a cofibration in this case. 
\end{definition}

We will often use the following description of mapping spaces in the over/under-category ${\cal P}_{(R, P)//(C, K)}$: the mapping space ${\rm Map}_{{\cal P}_{(R, P) // (C, K)}}((A, M), (B, N))$ arises as the pullback of \begin{equation}\label{relativemappingspace} {\rm Map}_{{\cal C}{\cal S}^{\cal J}_{P // K}}(M, N) \xrightarrow{} {\rm Map}_{{\cal C}{\rm Sp}^{\Sigma}_{{\Bbb S}^{\cal J}[P] // C}}({\Bbb S}^{\cal J}[M], B) \xleftarrow{} {\rm Map}_{{\cal C}{\rm Sp}^{\Sigma}_{{\Bbb S}^{\cal J}[P] // C}}(A, B),\end{equation} where the morphisms are induced by the structure maps.

\section{Logarithmic topological Hochschild homology}\label{logthh} We now introduce a variant of logarithmic topological Hochschild homology relative to a map $(R, P) \to (A, M)$ of pre-logarithmic ring spectra. On one hand, we prove that it enjoys properties analogous to those of the relative topological Hochschild homology ${\rm THH}^R(A)$. On the other we prove that it enjoys properties analogous to those of the absolute construction ${\rm THH}(A, M)$ of \cite{RSS15}. 

\begin{definition}\label{def:logthh} Let $(R, P) \to (A, M)$ be a cofibration of cofibrant pre-logarithmic ring spectra. The \emph{logarithmic topological Hochschild homology} ${\rm THH}^{(R, P)}(A, M)$ is the commutative symmetric ring spectrum given by the pushout \[\begin{tikzcd}[row sep = small]{\Bbb S}^{\cal J}[B^{\rm cy}_P(M)] \ar{r} \ar{d} & {\Bbb S}^{\cal J}[B^{\rm cy}_P(M)^{\rm rep}] \ar{d} \\ {\rm THH}^R(A) \ar{r} & {\rm THH}^{(R, P)}(A, M)\end{tikzcd}\] of commutative symmetric ring spectra.
\end{definition}

We remark that ${\rm THH}^{(R, P)}(A, M)$ is naturally an object of the category ${\cal C}{\rm Sp}^{\Sigma}_{A // A}$ of augmented commutative $A$-algebras. 

\begin{remark}[The circle action on logarithmic ${\rm THH}$]\label{logthhcircle} In the above definition we have chosen to model the replete bar construction by a relative fibrant replacement in the group completion model structure. Equivalently, one could have defined the ``relative" replete bar construction by a (homotopy) cartesian square \[\begin{tikzcd}[row sep = small]B^{\rm rep}_P(M) \ar{d} \ar{r} & B^{\rm cy}_{P^{\rm gp}}(M^{\rm gp}) \ar{d} \\ M' \ar[two heads]{r} & M^{\rm gp},\end{tikzcd}\] where $M \xrightarrow{\simeq} M' \xrightarrow{} M^{\rm gp}$ is factorization in the positive ${\cal J}$-model structure on commutative ${\cal J}$-space monoids. By Lemma \ref{chainovergp} there is a chain of equivalences $B^{\rm cy}_{P^{\rm gp}}(M^{\rm gp}) \simeq (B^{\rm cy}_{P}(M))^{\rm gp}$, and as such it follows from virtual surjectivity of the augmentation to $M^{\rm gp}$ that $B^{\rm rep}_P(M)$ and $B^{\rm cy}_P(M)^{\rm rep}$ are ${\cal J}$-equivalent, see Lemma \ref{repletepullback}. One advantage that the replete bar construction $B^{\rm rep}_P(M)$ enjoys over the repletion $B^{\rm cy}_P(M)^{\rm rep}$ is that it inherits a cyclic action from the cyclic bar construction $B^{\rm cy}_{P^{\rm gp}}(M^{\rm gp})$.  

While we are not free to exchange $B^{\rm cy}_P(M)^{\rm rep}$ with $B^{\rm rep}_P(M)$ in our definition of ${\rm THH}^{(R, P)}(A, M)$ since the map $B^{\rm cy}_P(M) \to B^{\rm rep}_P(M)$ may fail to be a cofibration, it is easy to see that our definition is weakly equivalent to that of \cite{RSS15} in the case of absolute case of ${\rm THH}(A, M)$. \end{remark}

\subsection{The relation between absolute and relative log ${\rm THH}$} We now provide an analogue of the isomorphism $P \boxtimes_{B^{\rm cy}(P)} B^{\rm cy}(M) \cong B^{\rm cy}_P(M)$ from Proposition \ref{prop:cycbar} for logarithmic ${\rm THH}$. To ensure that the relevant balanced smash product captures a well-defined homotopy type, we form a cofibrant replacement \begin{equation}\label{projcof} B^{\rm cy}(A)^{\rm c} \xleftarrow{} {\Bbb S}^{\cal J}[B^{\rm cy}(M)]^{\rm c} \xrightarrow{} {\Bbb S}^{\cal J}[B^{\rm cy}(M)^{\rm rep}]^{\rm c}\end{equation} of the $(* \xleftarrow{} * \xrightarrow{} *)$-shaped diagram defining ${\rm THH}(A, M)$ relative to that defining ${\rm THH}(R, P)$ in the projective model structure (see \cite[Proposition 10.6]{DS95}). We shall denote by ${\rm THH}(A, M)^{\rm c}$ the pushout of the diagram (\ref{projcof}). There is a commutative diagram  \begin{equation}\label{projcof33}\begin{tikzcd}[row sep = small]R & \ar{l} B^{\rm cy}(R) \ar{r} & B^{\rm cy}(A)^{\rm c} \\ {\Bbb S}^{\cal J}[P] \ar{u} \ar{d}{=} &  {\Bbb S}^{\cal J}[B^{\rm cy}(P)] \ar{l} \ar{r} \ar{u} \ar{d} & {\Bbb S}^{\cal J}[B^{\rm cy}(M)]^{\rm c} \ar{u} \ar{d} \\ {\Bbb S}^{\cal J}[P] &  {\Bbb S}^{\cal J}[B^{\rm cy}(P)^{\rm rep}] \ar{l} \ar{r} & {\Bbb S}^{\cal J}[B^{\rm cy}(M)^{\rm rep}]^{\rm c}\end{tikzcd}\end{equation} of commutative symmetric ring spectra. We shall denote by ${\rm THH}^{(R, P)}(A, M)^{\rm c}$ the commutative symmetric ring spectrum obtained as the colimit of the diagram (\ref{projcof33}) by  first forming the horizontal pushouts. 

\begin{lemma}\label{logthhjuggling} There is a chain of stable equivalences \[R \wedge_{{\rm THH}(R, P)} {\rm THH}(A, M)^{\rm c} \xrightarrow{\simeq} {\rm THH}^{(R, P)}(A, M)^{\rm c} \xleftarrow{\simeq} {\rm THH}^{(R, P)}(A, M)\] of commutative symmetric ring spectra. 
\end{lemma}

\begin{proof}  This follows from commuting homotopy pushouts in the diagram (\ref{projcof33}), Proposition \ref{prop:cycbar} and Lemma \ref{repbarjuggling}. 
\end{proof}

\subsection{Logification invariance of relative log ${\rm THH}$} We now discuss the effect of logification on log ${\rm THH}$:

\begin{proposition}\label{loginvariance} The logification construction induces stable equivalences \[{\rm THH}^{(R, P)}(A, M) \xrightarrow{\simeq} {\rm THH}^{(R, P)}(A, M^a) \xrightarrow{\simeq} {\rm THH}^{(R, P^a)}(A, M^a)\] of commutative symmetric ring spectra. 
\end{proposition}

In the absolute setting we remark that, unlike in the context of \cite[Section 4.3]{RSS15}, it is not necessary to form an additional cofibrant replacement of $A$ when passing from ${\rm THH}(A, M)$ to ${\rm THH}(A, M^a)$, as the pushout of \[{\rm THH}(A) \xleftarrow{} {\Bbb S}^{\cal J}[B^{\rm cy}(M)] \xrightarrow{} {\Bbb S}^{\cal J}[B^{\rm cy}(M)^{\rm rep}]\] defining ${\rm THH}(A, M)$ captures a well-defined homotopy type without further assumptions on $A$.

\begin{proof} In the case of absolute log ${\rm THH}$, this is \cite[Theorem 4.24]{RSS15}. We reduce from the relative to the absolute case using Lemma \ref{logthhjuggling}: there is a commutative diagram \[\begin{tikzcd}[row sep = small]R \wedge_{{\rm THH}(R, P)} {\rm THH}(A, M)^{\rm c} \ar{r}{\simeq} \ar{d}{\simeq} & {\rm THH}^{(R, P)}(A, M)^{\rm c} \ar{d} & {\rm THH}^{(R, P)}(A, M)  \ar[swap]{l}{\simeq} \ar{d} \\ R \wedge_{{\rm THH}(R, P)} {\rm THH}(A, M^a)^{\rm c} \ar{d}{\simeq} \ar{r}{\simeq} & {\rm THH}^{(R, P)}(A, M^a)^{\rm c} \ar{d} & {\rm THH}^{(R, P)}(A, M^a) \ar{d} \ar[swap]{l}{\simeq} \\ R \wedge_{{\rm THH}(R, P^a)} {\rm THH}(A, M^a)^{\rm c} \ar{r}{\simeq} & {\rm THH}^{(R, P^a)}(A, M^a)^{\rm c} & {\rm THH}^{(R, P)}(A, M^a) \ar[swap]{l}{\simeq}\end{tikzcd}\] of commutative symmetric ring spectra, in which the indicated morphisms are stable equivalences. We obtain the desired statement by the two-out-of-three property. 
\end{proof}

\subsection{Logarithmic ${\rm THH}$ as a suspension}

Proposition \ref{prop:repbarsuspension} describes the replete bar construction $B^{\rm cy}_P(M)^{\rm rep}$ as a suspension in the category of augmented commutative ${\cal J}$-space monoids, while Proposition \ref{thhsuspension} describes the ordinary topological Hochschild homology ${\rm THH}^R(A)$ as a suspension in the category of augmented commutative $A$-algebras. Gluing these facts together, we can prove the analogous statement for logarithmic topological Hochschild homology:

\begin{proposition}\label{logthhsuspension} Let $(R, P) \to (A, M)$ be a cofibration of cofibrant pre-logarithmic ring spectra. There is a chain of stable equivalences of augmented commutative $A$-algebras relating ${\rm THH}^{(R, P)}(A, M)$ and \begin{equation}\label{suspensioninlogthh}S^1 \odot_A ((A \wedge_R A) \wedge_{{\Bbb S}^{\cal J}[M \boxtimes_P M]} {\Bbb S}^{\cal J}[(M \boxtimes_P M)^{\rm rep}]),\end{equation} the suspension of $(A \wedge_R A) \wedge_{{\Bbb S}^{\cal J}[M \boxtimes_P M]} {\Bbb S}^{\cal J}[(M \boxtimes_P M)^{\rm rep}]$ in the category of augmented commutative $A$-algebras.
\end{proposition}

\begin{proof} Keeping Definition \ref{tensors} in mind, we consider the commutative diagram \[\begin{tikzcd}[row sep = small]A & A \wedge_R A \ar{r} \ar{l} & S^1 \otimes_A (A \wedge_R A) \\ {\Bbb S}^{\cal J}[M] \ar{u} \ar{d}{=} & {\Bbb S}^{\cal J}[M \boxtimes_P M] \ar{u} \ar{d} \ar{l} \ar{r} & S^1 \otimes_{{\Bbb S}^{\cal J}[M]} {\Bbb S}^{\cal J}[M \boxtimes_P M] \ar{u} \ar{d} \\ {\Bbb S}^{\cal J}[M] & {\Bbb S}^{\cal J}[(M \boxtimes_P M)^{\rm rep}] \ar{l} \ar{r} & S^1 \otimes_{{\Bbb S}^{\cal J}[M]} {\Bbb S}^{\cal J}[(M \boxtimes_P M)^{\rm rep}]\end{tikzcd}\] of commutative symmetric ring spectra. Commuting homotopy pushouts reveals that the suspension (\ref{suspensioninlogthh}) is naturally stably equivalent to \begin{equation}\label{coprodofsuspensions}(S^1 \odot_A (A \wedge_R A)) \wedge_{S^1 \odot_{{\Bbb S}^{\cal J}[M]} {\Bbb S}^{\cal J}[M \boxtimes_P M]} (S^1 \odot_{{\Bbb S}^{\cal J}[M]} {\Bbb S}^{\cal J}[(M \boxtimes_P M)^{\rm rep}]).\end{equation} There are isomorphisms $S^1 \odot_A (A \wedge_R A) \cong B^{\rm cy}_R(A)$ and $S^1 \odot_{{\Bbb S}^{\cal J}[M]} {\Bbb S}^{\cal J}[M \boxtimes_P M] \cong B^{\rm cy}_{{\Bbb S}^{\cal J}[P]}({\Bbb S}^{\cal J}[M])$ by Proposition \ref{prop:cycbar}. Moreover, there is a chain of stable equivalences \[S^1 \odot_{{\Bbb S}^{\cal J}[M]} {\Bbb S}^{\cal J}[(M \boxtimes_P M)^{\rm rep}] \simeq {\Bbb S}^{\cal J}[S^1 \odot_M (M \boxtimes_P M)^{\rm rep}] \simeq {\Bbb S}^{\cal J}[B^{\rm cy}_P(M)^{\rm rep}],\] where the second chain of stable equivalences arises from applying ${\Bbb S}^{\cal J}[-]$ to the chain of ${\cal J}$-equivalences of Proposition \ref{prop:repbarsuspension}. The maps remain equivalences after applying ${\Bbb S}^{\cal J}[-]$, as all maps in the chain are augmented over the cofibrant commutative ${\cal J}$-space monoid $M$, so that \cite[Corollary 8.8]{RSS15} is applicable. Moreover, since the maps in the chain are under $B^{\rm cy}_P(M)$, they induce a chain of equivalences relating (\ref{coprodofsuspensions}) to $B^{\rm cy}_R(A) \wedge_{{\Bbb S}^{\cal J}[B^{\rm cy}_P(M)]} {\Bbb S}^{\cal J}[B^{\rm cy}_P(M)^{\rm rep}] = {\rm THH}^{(R, P)}(A, M)$, which concludes the proof. 
\end{proof}

\section{Logarithmic topological Andr\'e--Quillen homology}\label{logtaq} We now proceed to review the notion of logarithmic derivations following \cite{Sag14}, before we introduce our new definition of log ${\rm TAQ}$. This section contains many constructions involving mapping spaces in comma categories ${\cal C}_{x // y}$. Our preferred notation for these mapping spaces is ${\rm Map}{}_{{\cal C}_{x // y}}(-, -)$. When the category ${\cal C}$ is clear from context, we will on occasion shorten this to ${\rm Map}{}^{x/}_{/y}(-, -)$. 

\subsection{Derivations and ${\rm TAQ}$}

Let $A$ be a positive fibrant commutative symmetric ring spectrum and let $X$ be an $A$-module. Then we can form the \emph{square-zero} extension $A \vee X$: this has a multiplication coming from that on $A$, the $A$-module structure on $X$ and the trivial map $X \wedge X \to *$. This construction comes with a natural augmentation $A \vee X \to A$. To ensure that various mapping spaces capture the correct homotopy type, we fibrantly replace $A \vee X$ over $A$. We borrow the following notation from \cite[Definition 5.2]{Sag14}:

\begin{definition}\label{squarezero} We let $A \vee_f X$ denote a fibrant replacement over $A$ in the positive model structure on commutative symmetric ring spectra: \[\begin{tikzcd}A \vee X \ar[tail]{r}{\simeq} & A \vee_f X \ar[two heads]{r} & A.\end{tikzcd}\]
\end{definition}

Suppose now that $R \to A$ is a cofibration of commutative symmetric ring spectra. Then the \emph{space of $R$-algebra derivations from $A$ to $X$} is the mapping space \[{\rm Der}_R(A, X) := {\rm Map}_{{\cal C}{\rm Sp}^{\Sigma}_{R // A}}(A, A \vee_f X).\] In analogy with the situation in ordinary algebra, where derivations are corepresented by the module of K\"ahler differentials, the space ${\rm Der}_R(A, X)$ is corepresented by the $A$-module ${\rm TAQ}^R(A)$, the topological Andr\'e--Quillen homology of $A$, whose definition we now briefly recall. All statements made here are well-known and were first proven in \cite{Bas99}. 

The (already derived, by our cofibrancy hypothesis) smash product $A \wedge_R A$ is an augmented commutative $A$-algebra, with augmentation map the multiplication $A \wedge_R A \to A$. One can form the augmentation ideal $I_A(A \wedge_R A)$ as the non-unital commutative $A$-algebra arising as the point-set fiber of the augmentation map. This functorial procedure is the right adjoint in a Quillen equivalence between the categories of non-unital commutative $A$-algebras and augmented commutative $A$-algebras; the left adjoint is given by formally adding a unit: \[\begin{tikzcd}{\rm Nuca}_A \ar[shift left = 1]{r}{A \vee -} &  \ar[shift left = 1]{l}{I_A} \vspace{10 mm} {\cal C}{\rm Sp}^{\Sigma}_{A // A}\end{tikzcd}\] Moreover, given any non-unital commutative $A$-algebra $N$, one can form the $A$-module of \emph{indecomposables} $Q_A(N)$, defined as the point-set cofiber of the multiplication map. This construction is the left adjoint in a Quillen adjunction between the categories of $A$-modules and non-unital commutative $A$-algebras, where the right adjoint is given by considering any $A$-module as a non-unital commutative $A$-algebra with trivial multiplication. In conclusion, there are Quillen adjunctions \begin{equation}\label{basterraadj}\begin{tikzcd}\Mod_A \ar[shift right = 1]{r} & \ar[shift right = 1,swap]{l}{Q_A} \vspace{10 mm} {\rm Nuca}_A \ar[shift left = 1]{r}{A \vee -} &  \ar[shift left = 1]{l}{I_A} \vspace{10 mm} {\cal C}{\rm Sp}^{\Sigma}_{A // A}\end{tikzcd}\end{equation} with left adjoints on top, and the right-hand adjunction is a Quillen equivalence.

\begin{definition}\label{def:taq} Let $B$ be an augmented commutative $A$-algebra, and define the $A$-module \[{\rm taq}^A(B) := Q^{\Bbb L}_AI^{\Bbb R}_A(B)\] by evaluating the composite of the derived functors $I^{\Bbb R}$ and $Q^{\Bbb L}$ at $B$. If $R \to A$ is a cofibration of cofibrant commutative symmetric ring spectra, we define the \emph{topological Andr\'e--Quillen homology} of $A$ relative to $R$ to be the $A$-module \[{\rm TAQ}^R(A) := {\rm taq}^A(A \wedge_R A),\] where $A \wedge_R A$ is considered an augmented commutative $A$-algebra via the multiplication map $A \wedge_R A \to A$. 
\end{definition}

\begin{proposition}\label{prop:taqcorepresents}\cite[Proposition 3.2]{Bas99} The space of $R$-algebra derivations from $A$ to $X$ is corepresented by ${\rm TAQ}^R(A)$: that is, there is a natural weak equivalence \[{\rm Map}_{{\rm Mod}_A}({\rm TAQ}^R(A), X) \simeq {\rm Der}_R(A, X).\]
\end{proposition}

\begin{proof}[Proof sketch] The adjunctions (\ref{basterraadj}), using that the right-hand adjunction is a Quillen equivalence, provides a natural weak equivalence \[{\rm Map}_{{\rm Mod}_A}({\rm TAQ}^R(A), X) \simeq {\rm Map}_{{\cal C}{\rm Sp}^{\Sigma}_{A//A}}(A \wedge_R A, A \vee_f X).\] The result follows by restriction of scalars. 
\end{proof}

The following change-of-rings lemma will be used on numerous occasions:

\begin{lemma}\label{lem:changeofrings} Let $C$ be a cofibrant commutative augmented $A$-algebra and let $A \to B$ be a cofibration of commutative symmetric ring spectra. Then the $B$-module spectra $B \wedge_A {\rm taq}^A(C)$ and ${\rm taq}^B(B \wedge_A C)$ are naturally weakly equivalent. 
\end{lemma}

\begin{proof} By restriction of scalars and the adjunctions (\ref{basterraadj}), there is a natural weak equivalence \[{\rm Map}_{{\rm Mod}_B}(B \wedge_A {\rm taq}^A(C), X) \simeq {\rm Map}_{{\cal C}{\rm Sp}^{\Sigma}_{A // A}}(C, A \vee_f X)\] for any fibrant $B$-module $X$, which can be considered an $A$-module via the map $A \to B$. Using the homotopy cartesian square \[\begin{tikzcd}[row sep = small]A \vee_f X \ar{r} \ar{d} & B \vee_f X \ar{d} \\ A \ar{r} & B\end{tikzcd}\] and extension of scalars, we infer a natural weak equivalence \[{\rm Map}_{{\cal C}{\rm Sp}^{\Sigma}_{A // A}}(C, A \vee_f X) \simeq {\rm Map}_{{\cal C}{\rm Sp}^{\Sigma}_{B // B}}(B \wedge_A C, B \vee_f X),\] from which the result follows from the adjunctions (\ref{basterraadj}).  
\end{proof}

We will also need the transitivity sequence for ${\rm TAQ}$ as established in \cite{Bas99}. 

\begin{proposition}\label{prop:transitivity}  Let $R \xrightarrow{f} A \xrightarrow{g} B$ be cofibrations of cofibrant commutative symmetric ring spectra. Then there is a homotopy cofiber sequence \[B \wedge_A {\rm TAQ}^R(A) \to {\rm TAQ}^R(B) \to {\rm TAQ}^A(B)\] of $B$-module spectra.
\end{proposition}

\begin{proof} This follows from observing that \[{\rm Der}_R(A, X) \xleftarrow{} {\rm Der}_R(B, X) \xleftarrow{} {\rm Der}_A(B, X)\] is a homotopy fiber sequence for any fibrant $B$-module $X$. 
\end{proof}

\subsection{Logarithmic derivations} Following \cite{Sag14}, we introduce derivations in the context of pre-logarithmic ring spectra. 

\begin{construction}[Square-zero extensions of pre-logarithmic ring spectra]\label{squarezeroconstruction} Let $(A, M)$ be a pre-logarithmic ring spectrum and let $X$ be an $A$-module. We may then form the square-zero extension $A \vee_f X$ as in Definition \ref{squarezero}. We define a pre-logarithmic structure $(M + X)^{\cal J} \to \Omega^{\cal J}(A \vee_f X)$ as follows: form the (homotopy) pullback \[\begin{tikzcd}[row sep = small](1 + X)^{\cal J} \ar{r} \ar{d} & {\rm GL}_1^{\cal J}(A \vee_f X) \ar{d} \\ U^{\cal J} \ar{r} & {\rm GL}_1^{\cal J}(A)\end{tikzcd}\] of commutative ${\cal J}$-space monoids, where we recall that $U^{\cal J}$ is the initial commutative ${\cal J}$-space monoid. The coproduct $M \boxtimes (1 + X)^{\cal J}$ admits a map to $\Omega^{\cal J}(A \vee_f X)$ induced by the two composites \[M \rightarrow{} \Omega^{\cal J}(A) \xrightarrow{} \Omega^{\cal J}(A \vee_f X) \text{  and  } (1 + X)^{\cal J} \xrightarrow{} {\rm GL}_1^{\cal J}( A \vee_f X) \xrightarrow{} \Omega^{\cal J}(A \vee_f X).\] We define $(M + X)^{\cal J}$ via a factorization \[\begin{tikzcd}M \boxtimes (1 + X)^{\cal J} \ar[tail]{r}{\simeq} & (M + X)^{\cal J} \ar[two heads]{r} & M\end{tikzcd}\] of this morphism in the positive ${\cal J}$-model structure, whose lifting properties provides the desired pre-logarithmic structure $(M + X)^{\cal J} \to \Omega^{\cal J}(A \vee_f X)$: \[\begin{tikzcd}[row sep = small]M \boxtimes (1 + X)^{\cal J} \ar{rr} \ar[tail]{d}{\simeq} && \Omega^{\cal J}(A \vee_f X) \ar[two heads]{d} \\ (M + X)^{\cal J} \ar[dashed]{urr} \ar[two heads]{r} & M \ar{r} & \Omega^{\cal J}(A).\end{tikzcd}\] Here we have used that $\Omega^{\cal J}(-)$ is right Quillen, so that $A \vee_f X \to A$ gives rise to a positive fibration ${\Omega}^{\cal J}(A \vee_f X) \to \Omega^{\cal J}(A)$.
\end{construction}

\begin{definition}\label{squarezeroprelog} Let $(A, M)$ be a pre-logarithmic ring spectrum and let $X$ be an $A$-module. The \emph{square-zero extension} of $(A, M)$ by $X$ is the pre-logarithmic ring spectrum $(A \vee_f X, (M + X)^{\cal J})$ of Construction \ref{squarezeroconstruction}. 
\end{definition}

\begin{definition} Let $(R, P) \to (A, M)$ be a morphism of pre-logarithmic ring spectra and let $X$ be an $A$-module. The \emph{space of logarithmic derivations} with values in $X$ is the mapping space \[{\rm Der}_{(R, P)}((A, M), X) := {\rm Map}_{{\cal P}_{(R, P) // (A, M)}}((A, M), (A \vee_f X, (M + X)^{\cal J})).\]
\end{definition}

We recover the usual notion of derivation by embedding the category of commutative symmetric ring spectra in the category of pre-logarithmic ring spectra by means of the trivial log structures. The above definition is analogous to that of logarithmic derivations of pre-log rings, which are corepresented by the module of logarithmic K\"ahler differentials \cite[Proposition 4.27]{Rog09}.

\subsection{Logarithmic ${\rm TAQ}$} 

We are now prepared to explain our new construction of log ${\rm TAQ}$. By the description given in (\ref{relativemappingspace}), the space ${\rm Der}_{(R, P)}((A, M), X)$ of logarithmic derivations fits in a homotopy cartesian square 
\begin{equation}\label{logderprepullback}\begin{tikzcd}[row sep = small]{\rm Der}_{(R, P)}((A, M), X) \ar{r} \ar{d} & {\rm Map}_{{\cal C}{\rm Sp}^{\Sigma}_{R // A}}(A, A \vee_f X) \ar{d} \\ {\rm Map}_{{\cal C}{\cal S}^{\cal J}_{P // M}}(M, (M + X)^{\cal J}) \ar{r} & {\rm Map}_{{\cal C}{\rm Sp}^{\Sigma}_{{\Bbb S}^{\cal J}[P] // A}}({\Bbb S}^{\cal J}[M], A \vee_f X).\end{tikzcd}\end{equation}

In \cite[Proposition 5.19]{Sag14}, it is proven that the lower left-hand mapping space is corepresented by (the connective spectrum associated to) a certain quotient of Segal $\Gamma$-spaces. As the two right-hand mapping spaces are corepresented by appropriate ${\rm TAQ}$-terms, one obtains an $A$-module spectrum, which we denote here by ${\widetilde{\rm TAQ}}{}^{(R, P)}(A, M)$, by forming the homotopy pushout of the corepresenting objects. By construction, this $A$-module spectrum corepresents logarithmic derivations.

We now propose a new definition of log ${\rm TAQ}$. This makes use of the functor ${\rm taq}^A$ of Definition \ref{def:taq}.

\begin{definition}\label{def:logtaq} Let $(R, P) \to (A, M)$ be a cofibration of cofibrant pre-logarithmic ring spectra. The \emph{logarithmic topological Andr\'e--Quillen homology} of $(A, M)$ relative to $(R, P)$ is the $A$-module spectrum \[{\rm TAQ}^{(R, P)}(A, M) := {\rm taq}^A((A \wedge_R A) \wedge_{{\Bbb S}^{\cal J}[M \boxtimes_P M]} {\Bbb S}^{\cal J}[(M \boxtimes_P M)^{\rm rep}]),\] where $(M \boxtimes_P M)^{\rm rep}$ denotes the repletion of the multiplication map $M \boxtimes_P M \to M$.
\end{definition}

We remark that ${\rm TAQ}^{(R, P)}(A, M)$ fits in a homotopy cocartesian square \[\begin{tikzcd}[row sep = small] A \wedge_{{\Bbb S}^{\cal J}[M]} {\rm taq}^{{\Bbb S}^{\cal J}[M]}({\Bbb S}^{\cal J}[M \boxtimes_P M]) \ar{r} \ar{d} & A \wedge_{{\Bbb S}^{\cal J}[M]} {\rm taq}^{{\Bbb S}^{\cal J}[M]}({\Bbb S}^{\cal J}[(M \boxtimes_P M)^{\rm rep}]) \ar{d} \\ {\rm taq}^A(A \wedge_R A) \ar{r} & {\rm TAQ}^{(R, P)}(A, M).\end{tikzcd}\] of $A$-module spectra, where the left-hand vertical map is by definition the natural map $A \wedge_{{\Bbb S}^{\cal J}[M]} {\rm TAQ}^{{\Bbb S}^{\cal J}[P]}({\Bbb S}^{\cal J}[M]) \to {\rm TAQ}^R(A)$. This bears a close resemblance to the defining homotopy cocartesian square of logarithmic ${\rm THH}$ from Definition \ref{def:logthh}. This gives rise to a close relationship between the two notions which we exploit in Section \ref{lastsection}. 

\begin{theorem}\label{thm:taqcorepresents} The $A$-module ${\rm TAQ}^{(R, P)}(A, M)$ corepresents logarithmic derivations. That is, there is a natural weak equivalence \[{\rm Map}_{{\rm Mod}_A}({\rm TAQ}^{(R, P)}(A, M), X) \simeq {\rm Der}_{(R, P)}((A, M), X)\] for any fibrant $A$-module $X$. In particular, the $A$-module ${\rm TAQ}^{(R, P)}(A, M)$ is naturally weakly equivalent to the version of log ${\rm TAQ}$ studied in \cite{Sag14}. 
\end{theorem}

\begin{remark}\label{rem:logkahler} After proving Theorem \ref{thm:taqcorepresents}, we learned that there is a close analogy between our description of log ${\rm TAQ}$ and a description of the log K\"ahler differentials studied by Kato--Saito \cite[Section 4]{KS04}. In the same way that the module of K\"ahler differentials $\Omega^1_{A|R}$ associated to a map of discrete rings arise as the conormal of the diagonal map ${\rm Spec}(A) \to {\rm Spec}(A \otimes_R A)$, the log K\"ahler differentials $\Omega^1_{(A, M)|(R, P)}$ often arise as the conormal of a \emph{log diagonal map} out of ${\rm Spec}(A)$. Unwinding their definitions, we find that the definition of the log diagonal can be phrased in terms of repletion: For a map of discrete pre-log rings $(R, P) \to (A, M)$, it corresponds precisely to the augmentation map \[(A \otimes_R A) \otimes_{{\Bbb Z}[M \oplus_P M]} {\Bbb Z}[(M \oplus_P M)^{\rm rep}] \to A.\] This perspective will be elaborated upon in forthcoming joint work with Binda--Park--{\O}stv{\ae}r, in which we study Hochschild--Kostant--Rosenberg-type results in the context of log schemes.
\end{remark}

\begin{proof}[Proof of Theorem \ref{thm:taqcorepresents}] By extension of scalars, the square (\ref{logderprepullback}) can be rewritten as 
\begin{equation}\label{logderpullback}\begin{tikzpicture}[baseline= (a).base]
\node[scale=.97] (a) at (0,0){\begin{tikzcd}[row sep = small]{\rm Der}_{(R, P)}((A, M), X) \ar{r} \ar{d} & {\rm Map}_{{\cal C}{\rm Sp}^{\Sigma}_{A // A}}(A \wedge_R A, A \vee_f X) \ar{d} \\ {\rm Map}_{{\cal C}{\cal S}^{\cal J}_{M // M}}(M \boxtimes_P M, (M + X)^{\cal J}) \ar{r} & {\rm Map}_{{\cal C}{\rm Sp}^{\Sigma}_{{\Bbb S}^{\cal J}[M] // A}}({\Bbb S}^{\cal J}[M \boxtimes_P M], A \vee_f X).\end{tikzcd}}; \end{tikzpicture}\end{equation}
\noindent Here we have used that ${\Bbb S}^{\cal J} \colon {\cal C}{\cal S}^{\cal J} \to {\cal C}{\rm Sp}^{\Sigma}$ is strong symmetric monoidal, so that ${\Bbb S}^{\cal J}[M] \wedge_{{\Bbb S}^{\cal J}[P]} {\Bbb S}^{\cal J}[M] \cong {\Bbb S}^{\cal J}[M \boxtimes_P M]$ as commutative symmetric ring spectra.

We know that the right-hand morphism in the diagram (\ref{logderpullback}) is corepresented by the morphism $A \wedge_{{\Bbb S}^{\cal J}[M]} {\rm TAQ}{}^{{\Bbb S}^{\cal J}[P]}({\Bbb S}^{\cal J}[M]) \to {\rm TAQ}^R(A)$ by Proposition \ref{prop:taqcorepresents} and Lemma \ref{lem:changeofrings}. Hence our only remaining task is two produce a chain of equivalences relating the mapping spaces \begin{equation}\label{twomappingspaces}{\rm Map}_{{\cal C}{\cal S}^{\cal J}_{M // M}}(M \boxtimes_P M, (M + X)^{\cal J}) \text{ and } {\rm Map}_{{\cal C}{\rm Sp}^{\Sigma}_{{\Bbb S}^{\cal J}[M] // A}}({\Bbb S}^{\cal J}[(M \boxtimes_P M)^{\rm rep}], A \vee_f X)\end{equation} which are compatible with the maps to ${\rm Map}{}_{{\cal C}{\rm Sp}^{\Sigma}_{{\Bbb S}^{\cal J}[M] // A}}({\Bbb S}^{\cal J}[M \boxtimes_P M], A \vee_f X)$. Indeed, using Lemma \ref{lem:changeofrings} it is easy to see that the latter of the two mapping spaces is corepresented by $A \wedge_{{\Bbb S}^{\cal J}[M]} {\rm taq}^{{\Bbb S}^{\cal J}[M]}({\Bbb S}^{\cal J}[(M \boxtimes_P M)^{\rm rep}])$, and so the result follows by considering homotopy cocartesian square of corepresenting objects associated to~(\ref{logderpullback}). 

By Lemmas \ref{repletemonoid}, \ref{taqrewritelemma1} and \ref{taqrewritelemma2} below and the paragraphs between them, we obtain the following diagram with the indicated weak equivalences: \begin{equation}\label{bestdiagram}\begin{tikzpicture}[baseline= (a).base]
\node[scale=.91] (a) at (0,0){\begin{tikzcd}[column sep = tiny] {\rm Map}^{M/}_{/M}(M \boxtimes_P M, (M + X)^{\cal J}) \ar{r} & {\rm Map}^{{\Bbb S}^{\cal J}[M]/}_{/A}({\Bbb S}^{\cal J}[M \boxtimes_P M], A \vee_f X) \\ {\rm Map}^{M/}_{/M}((M \boxtimes_P M)^{\rm rep}, (M + X)^{\cal J})  \ar{u}{\simeq} \ar[swap]{u}{\text{Lemma } \ref{repletemonoid}} \ar{r} & {\rm Map}^{{\Bbb S}^{\cal J}[M]/}_{/A}({\Bbb S}^{\cal J}[(M \boxtimes_P M)^{\rm rep}], A \vee_f X)  \ar{u} \\  {\rm Map}^{M/}_{/M}((M \boxtimes_P M)^{\rm rep}, (M + X)^{\cal J})\ar[swap]{d}{\simeq} \ar{d}{\text{Lemma } \ref{taqrewritelemma1}} \ar{r} \ar{u}{=} & {\rm Map}^{{\Bbb S}^{\cal J}[M]/}_{/{\Bbb S}^{\cal J}[M]}({\Bbb S}^{\cal J}[(M \boxtimes_P M)^{\rm rep}], {\Bbb S}^{\cal J}[M] \vee_f X) \ar[swap]{d}{\simeq} \ar{d}{\text{Lemma } \ref{taqrewritelemma1}} \ar{u}{\simeq} \ar[swap]{u}{\text{Pullback along } {\Bbb S}^{\cal J}[M] \to A}\\ {\rm Map}_{/M}(W, (M + X)^{\cal J}) \ar[swap]{d}{=} \ar{r} & {\rm Map}_{/{\Bbb S}^{\cal J}[M]}({\Bbb S}^{\cal J}[W], {\Bbb S}^{\cal J}[M] \vee_f X) \ar[swap]{d}{\simeq} \ar{d}{\text{Pullback along } {\Bbb S}^{\cal J}[M] \to {\Bbb S}^{\cal J}[M]^{\rm f}} \\  {\rm Map}_{/M}(W, (M + X)^{\cal J}) \ar{r}  & {\rm Map}_{/{\Bbb S}^{\cal J}[M]^{\rm f}}({\Bbb S}^{\cal J}[W], {\Bbb S}^{\cal J}[M]^{\rm f} \vee_f X) \\ {\rm Map}_{/U^{\cal J}}(W, (1 + X)^{\cal J}) \ar{u}{\simeq} \ar[swap]{u}{\text{Lemma } \ref{taqrewritelemma2}} \ar{r}{\simeq} \ar[swap]{r}{\text{Lemma } \ref{taqrewritelemma2}} & {\rm Map}_{/\Omega^{\cal J}({\Bbb S}^{\cal J}[M]^{\rm f})}(W, \Omega^{\cal J}({\Bbb S}^{\cal J}[M]^{\rm f} \vee_f X)). \ar{u}{\simeq} \ar[swap]{u}{({\Bbb S}^{\cal J}, \Omega^{\cal J})\text{-adjunction}}\end{tikzcd}};\end{tikzpicture}\end{equation} By the two-out-three-property, the lower horizontal map in the top square of the diagram (\ref{bestdiagram}) is therefore a weak equivalence. This provides a chain of weak equivalences relating the mapping spaces (\ref{twomappingspaces}), and commutativity of the top square in the diagram (\ref{bestdiagram}) gives the desired compatability. As we have previously reduced the theorem to the existence of such a chain of weak equivalences, this concludes the proof. \end{proof}

The following series of lemmas were used in the above proof.

\begin{lemma}\label{repletemonoid} Let $M$ be a cofibrant commutative ${\cal J}$-space monoid. Then the repletion map $(M + X)^{\cal J} \xrightarrow{} ((M + X)^{\cal J})^{\rm rep}$ over $M$ is a ${\cal J}$-equivalence. In particular, the repletion map $M \boxtimes_P M \to (M \boxtimes_P M)^{\rm rep}$ induces a weak equivalence \[ {\rm Map}^{M/}_{/M}((M \boxtimes_P M)^{\rm rep}, (M + X)^{\cal J}) \xrightarrow{\simeq} {\rm Map}^{M/}_{/M}(M \boxtimes_P M, (M + X)^{\cal J})\] of mapping spaces.
\end{lemma}

\begin{proof} The natural map $(M + X)^{\cal J} \to M$ is virtually surjective, as it arises as factorization of the projection \[M \boxtimes (1 + X)^{\cal J} \to M \boxtimes U^{\cal J} \cong M.\] By Lemma \ref{repletepullback} it suffices to show that the square \[\begin{tikzcd}[row sep = small] M \boxtimes (1 + X)^{\cal J} \ar{r} \ar{d} & (M \boxtimes (1 + X)^{\cal J})^{\rm gp} \ar{d} \\ M \ar{r} & M^{\rm gp}.\end{tikzcd}\]
is homotopy cartesian with respect to the positive ${\cal J}$-model structure. Since $M^{\rm gp}$ is cofibrant and $(1 + X)^{\cal J}$ is grouplike, there are ${\cal J}$-equivalences \[M^{\rm gp} \boxtimes (1 + X)^{\cal J} \xrightarrow{\simeq} M^{\rm gp} \boxtimes ((1 + X)^{\cal J})^{\rm gp} \xrightarrow{\simeq} (M \boxtimes (1 + X)^{\cal J})^{\rm gp},\] where the last ${\cal J}$-equivalence arises from Lemma \ref{gpcomppush}. In conclusion, the square which we wish to prove is homotopy cartesian is of the form \[\begin{tikzcd}[row sep = small] M \boxtimes (1 + X)^{\cal J} \ar{r} \ar{d} & M^{\rm gp} \boxtimes (1 + X)^{\cal J} \ar{d} \\ M \ar{r} & M^{\rm gp}\end{tikzcd}\] up to ${\cal J}$-equivalence. This square is homotopy cartesian precisely when the induced square on Bousfield--Kan homotopy colimits over ${\cal J}$ is \cite[Corollary 11.4]{SS12}, and so the first claim follows. This gives rise to a commutative diagram \[\begin{tikzcd}[row sep = small]{\rm Map}^{M/}_{/M}((M \boxtimes_P M)^{\rm rep}, (M + X)^{\cal J}) \ar{r} \ar{d}{\simeq} & {\rm Map}^{M/}_{/M}(M \boxtimes_P M, (M + X)^{\cal J}) \ar{d}{\simeq} \\ {\rm Map}^{M/}_{/M}((M \boxtimes_P M)^{\rm rep}, ((M + X)^{\cal J})^{\rm rep}) \ar{r}{\simeq} & {\rm Map}^{M/}_{/M}(M \boxtimes_P M, ((M + X)^{\cal J})^{\rm rep}).\end{tikzcd}\] The vertical maps are weak equivalences by the first part of the lemma. The lower horizontal map is a weak equivalence since $((M + X)^{\cal J})^{\rm rep}$ is fibrant over $M$ in the group completion model structure, in which the repletion map is a weak equivalence by definition. This gives the second statement.
\end{proof}

Pulling back along ${\Bbb S}^{\cal J}[M] \to A$ and arguing as in the proof of Lemma \ref{lem:changeofrings}, one obtains the second square from the top in the diagram (\ref{bestdiagram}). The following observation provides the third square. We remark that in the statement we consider $W(M \boxtimes_P M)$ (as defined in Definition \ref{w}) as a commutative ${\cal J}$-space monoid over $M$ via its augmentation to the initial commutative ${\cal J}$-space monoid $U^{\cal J}$. For brevity, we shall simply write $W$ for the commutative ${\cal J}$-space monoid $W(M \boxtimes_P M)$.

\begin{lemma}\label{taqrewritelemma1} The weak equivalence $M \boxtimes W \to (M \boxtimes_P M)^{\rm rep}$ of Lemma \ref{repleterewrite} and restriction of scalars induce a commutative diagram \[\begin{tikzpicture}[baseline= (a).base]
\node[scale=.97] (a) at (0,0){\begin{tikzcd}[column sep = small, row sep = small]{\rm Map}^{M/}_{/M}((M \boxtimes_P M)^{\rm rep}, (M + X)^{\cal J}) \ar{r} \ar{d} & {\rm Map}_{/M}(W, (M + X)^{\cal J}) \ar{d} \\ {\rm Map}^{{\Bbb S}^{\cal J}[M]/}_{/{\Bbb S}^{\cal J}[M]}({\Bbb S}^{\cal J}[(M \boxtimes_P M)^{\rm rep}], {\Bbb S}^{\cal J}[M] \vee_f X)  \ar{r} & {\rm Map}_{/{\Bbb S}^{\cal J}[M]}({\Bbb S}^{\cal J}[W], {\Bbb S}^{\cal J}[M] \vee_f X)\end{tikzcd}}; \end{tikzpicture}\] in which the horizontal maps are weak equivalences. \qed
\end{lemma}

Let ${\Bbb S}^{\cal J}[M]^{\rm f}$ denote a fibrant replacement of the commutative symmetric ring spectrum ${\Bbb S}^{\cal J}[M]$. The lifting properties of the positive projective model structure provides a map ${\Bbb S}^{\cal J}[M]^{\rm f} \to A$ under ${\Bbb S}^{\cal J}[M]$, under which we can consider the $A$-module $X$ as an ${\Bbb S}^{\cal J}[M]^{\rm f}$-module by restriction of scalars. This provides the fourth square from the top in the diagram (\ref{bestdiagram}) by pullback along ${\Bbb S}^{\cal J}[M] \to {\Bbb S}^{\cal J}[M]^{\rm f}$.

We now provide the bottom square in the diagram (\ref{bestdiagram}). For this we shall again exploit the fact that $W$ is a grouplike commutative ${\cal J}$-space monoid augmented over the initial object $U^{\cal J}$:

\begin{lemma}\label{taqrewritelemma2} There is a commutative solid arrow diagram
\[\begin{tikzcd}[row sep = small] {\rm Map}_{/M}(W, (M + X)^{\cal J}) \ar{r}{\simeq} \ar[dashed]{dr} & {\rm Map}_{/{\Bbb S}^{\cal J}[M]^{\rm f}}({\Bbb S}^{\cal J}[W], {\Bbb S}^{\cal J}[M]^{\rm f} \vee_f X) \\ {\rm Map}_{/U^{\cal J}}(W, (1 + X)^{\cal J}) \ar{u}{\simeq}  \ar{r}{\simeq}  & {\rm Map}_{/\Omega^{\cal J}({\Bbb S}^{\cal J}[M]^{\rm f})}(W, \Omega^{\cal J}({\Bbb S}^{\cal J}[M]^{\rm f} \vee_f X)) \ar{u}{\simeq} \end{tikzcd}\] in which all maps are weak equivalences. 
\end{lemma}

\begin{proof} As homotopy cartesian squares of commutative ${\cal J}$-space monoids are detected on Bousfield--Kan homotopy colimits \cite[Corollary 11.4]{SS12} and $U^{\cal J}_{h{\cal J}}$ is contractible, the square \[\begin{tikzcd}[row sep = small](1 + X)^{\cal J} \ar{d} \ar{r} & (M + X)^{\cal J} \ar{d} \\ U^{\cal J} \ar{r} & M\end{tikzcd}\] is homotopy cartesian, from which we infer the left-hand weak equivalence. By definition, there is a homotopy cartesian square \[\begin{tikzcd}[row sep = small](1 + X)^{\cal J} \ar{d} \ar{r} & {\rm GL}_1^{\cal J}({\Bbb S}^{\cal J}[M]^{\rm f} \vee_f X) \ar{d} \\ U^{\cal J} \ar{r} & {\rm GL}_1^{\cal J}({\Bbb S}^{\cal J}[M]^{\rm f}),\end{tikzcd}\] from which we infer that the map \begin{equation}\label{firstweq}{\rm Map}_{/U^{\cal J}}(W, (1 + X)^{\cal J}) \xrightarrow{\simeq} {\rm Map}_{/{\rm GL}_1^{\cal J}({\Bbb S}^{\cal J}[M]^{\rm f})}(W, {\rm GL}_1^{\cal J}({\Bbb S}^{\cal J}[M]^{\rm f} \vee_f X))\end{equation} is a weak equivalence. 

The functor which assigns to any commutative ${\cal J}$-space monoid $M$ its units $M^{\times}$ is a right adjoint of the inclusion of grouplike commutative ${\cal J}$-space monoids to all commutative ${\cal J}$-space monoids. Applying this to the situation at hand, the fact that $W$ is grouplike implies that the map \begin{equation}\label{secondweq}{\rm Map}_{/{\rm GL}_1^{\cal J}({\Bbb S}^{\cal J}[M]^{\rm f})}(W, {\rm GL}_1^{\cal J}({\Bbb S}^{\cal J}[M]^{\rm f} \vee_f X)) \xrightarrow{\simeq} {\rm Map}_{/\Omega^{\cal J}({\Bbb S}^{\cal J}[M]^{\rm f})}(W, \Omega^{\cal J}({\Bbb S}^{\cal J}[M]^{\rm f} \vee_f X))\end{equation} induced by the inclusion of units is a weak equivalence. Composing the maps (\ref{firstweq}) and (\ref{secondweq}) we obtain the bottom weak equivalence. 

We now recall from Construction \ref{squarezeroconstruction} that the morphism ${\Bbb S}^{\cal J}[(M + X)^{\cal J}] \to {\Bbb S}^{\cal J}[M]^{\rm f} \vee_f X$ arises as the adjoint of a morphism $(M + X)^{\cal J} \to \Omega^{\cal J}({\Bbb S}^{\cal J}[M]^{\rm f} \vee_f X)$, which induces the dashed morphism in the diagram. Hence the upper triangle cut out by the dashed arrow is commutative. Moreover, the morphism \[(1 + X)^{\cal J} \to \Omega^{\cal J}({\Bbb S}^{\cal J}[M]^{\rm f} \vee_f X)\] factors through $(M + X)^{\cal J}$ by construction, and so the lower triangle is commutative as well.
\end{proof}

\subsection{The transitivity sequence for log ${\rm TAQ}$} Logarithmic ${\rm TAQ}$ enjoys the following transitivity sequence:

\begin{proposition}\label{logtransitivity} Let $(R, P) \xrightarrow{(f, f^\flat)} (A, M) \xrightarrow{(g, g^\flat)} (B, N)$ be cofibrations of cofibrant pre-logarithmic ring spectra. There is a homotopy cofiber sequence \[B \wedge_A {\rm TAQ}^{(R, P)}(A, M) \xrightarrow{} {\rm TAQ}^{(R, P)}(B, N) \xrightarrow{} {\rm TAQ}^{(A, M)}(B, N)\] of $B$-module spectra.
\end{proposition}

The following argument is effectively that given in \cite[Proposition 11.28]{Rog09}.

\begin{proof} We claim that there is a homotopy fiber sequence \[{\rm Der}_{(R, P)}((A, M), X) \xleftarrow{} {\rm Der}_{(R, P)}((B, N), X) \xleftarrow{} {\rm Der}_{(A, M)}((B, N), X)\] for any fibrant $B$-module $X$, from which the result will follow by considering the corresponding cofiber sequence of corepresenting objects. By definition, the space of logarithmic derivations arises as a homotopy pullback in which one leg is a map of spaces corepresented by appropriate ${\rm TAQ}$-terms. Since we know that ordinary ${\rm TAQ}$ satisfies transitivity by Proposition \ref{prop:transitivity}, it suffices to show that the sequence \[\begin{tikzcd}[column sep = small]{\rm Map}^{P/}_{/M}(M, (M \!+\! X)^{\cal J}) & {\rm Map}^{P/}_{/N}(N, (N \!+\! X)^{\cal J}) \ar{l} & {\rm Map}^{M/}_{/N}(N, (N \!+\! X)^{\cal J}) \ar{l}\end{tikzcd}\] relating the involved mapping spaces of commutative ${\cal J}$-space monoids is a homotopy fiber sequence for any fibrant $B$-module $X$. Since the square \[\begin{tikzcd}[row sep = small](M + X)^{\cal J} \ar{r} \ar{d} & (N + X)^{\cal J} \ar{d} \\ M \ar{r} & N\end{tikzcd}\] is homotopy cartesian, there is a natural weak equivalence \[{\rm Map}_{{\cal C}{\cal S}^{\cal J}_{P // M}}(M, (M + X)^{\cal J}) \simeq {\rm Map}_{{\cal C}{\cal S}^{\cal J}_{P // N}}(M, (N + X)^{\cal J}).\] It follows that the sequence in question is a homotopy fiber sequence. \end{proof}

\subsection{Logification invariance of log ${\rm TAQ}$} It is proved in \cite[Corollary 6.7]{Sag14} that the version of log ${\rm TAQ}$ studied in \emph{loc.\ cit.}\ is logification invariant. As this construction is naturally weakly equivalent to ours, we obtain the following:

\begin{proposition}\label{prop:logtaqlogi} Let $(R, P) \to (A, M)$ be a cofibration of cofibrant pre-logarithmic ring spectra. The logification construction induces weak equivalences \[{\rm TAQ}^{(R, P)}(A, M) \xrightarrow{\simeq} {\rm TAQ}^{(R, P)}(A, M^a) \xrightarrow{\simeq} {\rm TAQ}^{(R, P^a)}(A, M^a)\] of $A$-modules. \qed
\end{proposition}

\section{The log \'etale descent formula}\label{lastsection} We prove Theorem \ref{mainthm} in a series of propositions, each of which we motivate with the analogous result for ordinary ${\rm THH}$.

\subsection{Log \'etale descent implies formally log ${\rm THH}$-\'etale}\label{etaledescentimpliesthhetale} Suppose $R \to A$ is a cofibration of cofibrant commutative symmetric ring spectra which satisfies \'etale descent, that is, the natural map \[A \wedge_R {\rm THH}(R) \xrightarrow{\simeq} {\rm THH}(A)\] is a stable equivalence. Then there are stable equivalences \[A \xrightarrow{\cong} R \wedge_{{\rm THH}(R)} ({\rm THH}(R) \wedge_R A) \xrightarrow{\simeq} R \wedge_{{\rm THH}(R)} {\rm THH}(A) \xrightarrow{\simeq} {\rm THH}^R(A),\] so that $R \to A$ is formally ${\rm THH}$-\'etale. Here the last isomorphism is Proposition \ref{thhsuspension}, while the second stable equivalence is due to \'etale descent.

\begin{proposition}\label{logetaledescentimplieslogthhetale} Any cofibration $(R, P) \xrightarrow{} (A, M)$ of cofibrant pre-logarithmic ring spectra satisfying log \'etale descent is formally log ${\rm THH}$-\'etale.
\end{proposition}

\begin{proof} The argument is analogous to that of the classical case: there are equivalences \begin{align*}A    \xrightarrow{\cong}   R \wedge_{{\rm THH}(R, P)}  ({\rm THH}(R, P)  \wedge_{R} A)  & \xrightarrow{\simeq} R \wedge_{{\rm THH}(R, P)} {\rm THH}(A, M)^{\rm c} \\ &  \xrightarrow{\simeq} {\rm THH}^{(R, P)}(A, M)^{\rm c} \xleftarrow{\simeq} {\rm THH}^{(R, P)}(A, M),\end{align*} where the second stable equivalence is due to log \'etale descent and the following chain is Lemma~\ref{logthhjuggling}. As all maps in the chain are under $A$, the result follows by the two-out-of-three property.
\end{proof}

\subsection{Formally log ${\rm THH}$-\'etale implies formally log ${\rm TAQ}$-\'etale}\label{thhetaleimpliestaqetale}  The key observation in relating log ${\rm TAQ}$ and log ${\rm THH}$ is the following:

\begin{lemma}\label{taqoflogthh} The $A$-modules \[{\rm taq}^A({\rm THH}^{(R, P)}(A, M)) \text{ and } \Sigma{\rm TAQ}^{(R, P)}(A, M)\] are naturally weakly equivalent. 
\end{lemma}

\begin{proof} Propostion \ref{logthhsuspension} realizes ${\rm THH}^{(R, P)}(A, M)$ as the suspension of the commutative augmented $A$-algebra \[(A \wedge_R A) \wedge_{{\Bbb S}^{\cal J}[M \boxtimes_P M]} {\Bbb S}^{\cal J}[(M \boxtimes_P M)^{\rm rep}].\] By definition, the functor ${\rm taq}^A(-)$ is the composite of a Quillen equivalence and a left Quillen functor, both of which commute with suspensions. This means that the $A$-modules \[{\rm taq}^A({\rm THH}^{(R, P)}(A, M)) \text{ and } \Sigma {\rm taq}^A((A \wedge_R A) \wedge_{{\Bbb S}^{\cal J}[M \boxtimes_P M]} {\Bbb S}^{\cal J}[(M \boxtimes_P M)^{\rm rep}])\] are naturally weakly equivalent. 
\end{proof}

\begin{proposition}\label{logthhetaleimplieslogtaqetale} Let $(R, P) \xrightarrow{(f, f^{\flat})} (A, M)$ be a formally log ${\rm THH}$-\'etale morphism. Then $(f, f^{\flat})$ is also formally log ${\rm TAQ}$-\'etale.
\end{proposition}

\begin{proof}  By assumption, the unit map $A \xrightarrow{\simeq} {\rm THH}^{(R, P)}(A, M)$ is a stable equivalence. The result follows from applying ${\rm taq}^A(-)$ to this map and from Lemma \ref{taqoflogthh}. \qedhere
\end{proof}

\subsection{Formally log ${\rm TAQ}$-\'etale implies log \'etale descent} We finally discuss the log \'etale descent formula under the hypothesis of contractible log ${\rm TAQ}$. As we noted in the introduction, this fails already for ordinary ${\rm THH}$ unless one adds connectivity hypotheses. In the setting of pre-logarithmic ring spectra, we employ the following notion of connectivity:

\begin{definition}\label{logconnective} A pre-logarithmic ring spectrum $(A, M)$ is \emph{connective} if both the underlying commutative symmetric ring spectrum $A$ and ${\Bbb S}^{\cal J}[M]$ are connective. 
\end{definition}

For example, the pre-logarithmic ring spectra $(A, D(x))$ discussed in Example \ref{ex:prelogstr}(5) are connective. Our definition of connectivity is made so that the full strength of the following result, which is a reformulation of \cite[Theorem 6.10]{Kuh06}, will be applicable:

\begin{theorem}\label{kuhnanalysis} Let $B \to C \to B$ be an augmented commutative $B$-algebra. There is a natural tower of fibrations \[\cdots \to P_{B, 2}(C) \to P_{B, 1}(C) \to P_{B, 0}(C)\] of augmented commutative $B$-algebras satisying the following properties:

\begin{enumerate}
\item There is a weak equivalence $P_{B, 0}(C) \simeq B$, under which the map $C \to P_{B, 0}(C)$ corresponds to the augmentation map $C \to B$. 
\item The fiber of the fibration $P_{B, n}(C) \to P_{B, n - 1}(C)$ is weakly equivalent to the extended powers \[[\bigwedge_B^n {\rm taq}^B(C)]_{h\Sigma_n}\] as $B$-modules.
\item If $I_B^{\Bbb R}(C)$ is $0$-connected, then the map $C \to P_{B, n}(C)$ is $n$-connected. 
\end{enumerate}
\end{theorem}

For example, the above result applies for the augmented commutative symmetric ring spectrum \[A \to {\rm THH}(A) \to A\] for $A$ connective. In this case, the augmentation map is an isomorphism on $\pi_0$ and, having a section, a surjection on $\pi_1$. Hence the above theorem applies to describe ${\rm THH}(A)$ as the homotopy limit of the tower $\{P_{A, n}({\rm THH}(A))\}$. Moreover, we can describe the homotopy fibers of the maps in the tower: applying Lemma \ref{taqoflogthh} in the setting of ordinary ${\rm THH}$ gives that these are merely extended powers of a suspension of ${\rm TAQ}(A)$. We will show that, under the connectivity hypothesis described in Definition \ref{logconnective}, we may draw similar conclusions about log ${\rm THH}$. 

The above formulation of Theorem \ref{kuhnanalysis} differs from the one given by Kuhn, as he employs a different formulation of topological Andr\'e--Quillen homology. The equivalence between the various notions is already alluded to in \emph{loc.\ cit.}, and is by now a well-known consequence of the work of Basterra and Mandell \cite{BM05}. As the present work is highly dependent upon the above formulation, we provide a proof explaining how one may pass between the two different setups. The author first learned of the below line of argument from a discussion on MathOverflow between Yonatan Harpaz and Bruno Stonek.\footnote{ \url{https://mathoverflow.net/questions/316418/}} 

\begin{proof}[Proof of Theorem \ref{kuhnanalysis}] By \cite[Theorems 3 and 4]{BM05}, the stabilization of the adjunctions (\ref{basterraadj}) gives rise to a chain of Quillen equivalences \begin{equation}\label{basterramandelleq}\begin{tikzcd}{\rm Mod}_B \ar[shift left = 1]{r}{\Sigma^{\infty}} & {\rm Sp}(\Mod_B) \ar[shift left = 1]{l} \ar[shift right = 1]{r} & \ar[shift right = 1,swap]{l}{Q_B} \vspace{10 mm} {\rm Sp}({\rm Nuca}_B) \ar[shift left = 1]{r}{B \vee -} &  \ar[shift left = 1]{l}{I_B} \vspace{10 mm} {\rm Sp}({\cal C}{\rm Sp}^{\Sigma}_{B // B}).\end{tikzcd}\end{equation} By \cite[Theorem 3.10]{Kuh06}, there is a tower of fibrations of commutative augmented $B$-algebras satisfying properties (1) and (3). Moreover, the fibers of the maps in the tower are described as extended powers of the $B$-module underlying the non-unital commutative $B$-algebra  \[{\rm taq}_B(C) := {\rm hocolim}_n \Omega^n(S^n \odot I^{\Bbb R}_B(C)),\] where the tensor $S^n \odot -$ is taken in the pointed model category of non-unital commutative $B$-algebras. This is the $0$th level of an $\Omega$-spectrum replacement of the suspension spectrum $\{S^n \odot I^{\Bbb R}_B(C)\}$ in the stable category ${\rm Sp}({\rm Nuca}_B)$. By \cite[Proposition 3.8]{BM05}, it makes homotopically no difference whether one applies the levelwise indecomposables $Q_B$ or the levelwise forgetful functor ${\rm Sp}({\rm Nuca}_B) \to {\rm Sp}({\rm Mod}_A) \simeq {\rm Mod}_A$. In particular, ${\rm taq}^B(C)$ is naturally weakly equivalent to the $B$-module underlying ${\rm taq}_B(C)$, which concludes the proof.
\end{proof}

\begin{remark} Since the augmentation ideal functor $I_B$ is the right adjoint in a Quillen equivalence, we also have the stabilization formula \[{\rm taq}^B(C) \simeq {\rm hocolim}_n \Omega^n(I^{\Bbb R}_B(S^n \odot_B C)),\] involving instead the pointed tensor in augmented commutative $B$-algebras. Setting $B = A$ and $C = (A \wedge_R A) \wedge_{{\Bbb S}^{\cal J}[M \boxtimes_P M]} {\Bbb S}^{\cal J}[(M \boxtimes_P M)^{\rm rep}]$, this gives a stabilization formula for log ${\rm TAQ}$. 
\end{remark}

The following result provides the necessary connectivity property for the augmentation ideal of ${\rm THH}^{(R, P)}(A, M) \to A$ for Theorem \ref{kuhnanalysis} to be applicable in the context of log ${\rm THH}$. In its proof we use the following basic consequence of the ${\rm Tor}$-spectral sequence \cite[Theorem IV.4.1]{EKMM}, which is for example spelled out in \cite[Corollary 7.2.1.23]{Lur}: if $R$ is a connective ring spectrum and $X$ and $Y$ are connective $R$-module spectra, then there is a natural isomorphism \begin{equation}\label{pizeroiso}\pi_0(X \wedge^{\Bbb L}_R Y) \cong \pi_0(X) \otimes_{\pi_0(R)} \pi_0(Y).\end{equation} 

\begin{proposition}\label{logaugmentation} Let $(R, P) \to (A, M)$ be a cofibration of cofibrant and connective pre-logarithmic ring spectra. Then the augmentation ${\rm THH}^{(R, P)}(A, M) \to A$ induces an isomorphism $\pi_0{\rm THH}^{(R, P)}(A, M) \xrightarrow{\cong} \pi_0A$ of commutative rings. 
\end{proposition}

\begin{proof} The chain of stable equivalences \[R \wedge_{{\rm THH}(R, P)} {\rm THH}(A, M)^{\rm c} \xrightarrow{\simeq} {\rm THH}^{(R, P)}(A, M)^{\rm c} \xleftarrow{\simeq} {\rm THH}^{(R, P)}(A, M)\] from Proposition \ref{logthhjuggling} and the isomorphism (\ref{pizeroiso}) allow us to reduce to the case of absolute log ${\rm THH}$. Moreover, the definition of log ${\rm THH}$ as a (derived) balanced smash product \[{\rm THH}(A, M) = {\rm THH}(A) \wedge_{{\Bbb S}^{\cal J}[B^{\rm cy}(M)]} {\Bbb S}^{\cal J}[B^{\rm cy}(M)^{\rm rep}]\] allows us, by another application of the isomorphism (\ref{pizeroiso}), to further reduce to checking that the augmentation \[{\Bbb S}^{\cal J}[B^{\rm cy}(M)^{\rm rep}] \xrightarrow{} {\Bbb S}^{\cal J}[M]\] is an isomorphism on $\pi_0$. We have used here that ${\Bbb S}^{\cal J}[B^{\rm cy}(M)] \cong {\rm THH}({\Bbb S}^{\cal J}[M])$ and that the commutative symmetric ring spectrum ${\Bbb S}^{\cal J}[M]$ is connective.

We study the analysis of the replete bar construction from \cite{RSS15} and \cite{RSS18}. Fix a factorization \[\begin{tikzcd}M \ar[tail]{r}{\simeq} & M' \ar[two heads]{r} & M^{\rm gp}\end{tikzcd}\] in the positive ${\cal J}$-model structure of the group completion $\eta_M \colon M \to M^{\rm gp}$. By \cite[Proposition 3.15]{RSS15}, there is a chain of ${\cal J}$-equivalences over $M'$ relating $B^{\rm cy}(M)^{\rm rep}$ to the (homotopy) pullback of \[M' \xrightarrow{} M^{\rm gp} \xleftarrow{} B^{\rm cy}(M^{\rm gp});\] this homotopy pullback is simply referred to as the \emph{replete bar construction} $B^{\rm rep}(M)$ in \cites{RSS15, RSS18}. By \cite[Proof of Proposition 3.1]{RSS18}, there is a ${\cal J}$-equivalence \[M \boxtimes V(M) \xrightarrow{\simeq} B^{\rm rep}(M)\] over $M'$, where, following \cite{RSS18}, we define $V(M)$ as the (homotopy) pullback of the diagram \[U(M^{\rm gp}) \xrightarrow{} M^{\rm gp} \xleftarrow{} B^{\rm cy}(M^{\rm gp}).\] Here $U(M^{\rm gp})$ appears in a factorization $U^{\cal J} \to U(M^{\rm gp}) \to M^{\rm gp}$ of the initial map to $M^{\rm gp}$ by an acyclic cofibration followed by a positive ${\cal J}$-fibration. 

By construction, the commutative ${\cal J}$-space monoid $M'$ is cofibrant, so that \cite[Corollary 8.8]{RSS15} applies to infer that \[ {\Bbb S}^{\cal J}[M \boxtimes V(M)] \xrightarrow{\simeq} {\Bbb S}^{\cal J}[B^{\rm rep}(M)]\] is a stable equivalence of commutative symmetric ring spectra. We shall argue that the augmentation \[{\Bbb S}^{\cal J}[M \boxtimes V(M)] \xrightarrow{} {\Bbb S}^{\cal J}[M] \wedge {\Bbb S}^{\cal J}[U(M^{\rm gp})] \xleftarrow{\simeq} {\Bbb S}^{\cal J}[M]\] induces an isomorphism on $\pi_0$. By \cite[Proposition 2.4]{RSS18}, there is a chain of equivalences relating $V(M)$ to $F_{({\bf 0, 0})}^{\cal J}(B(M^{\rm gp}_{h{\cal J}}))$; the proof in \emph{loc.\ cit} shows that the weak equivalences involved respect the augmentations to $U(M^{\rm gp})$. In particular, it suffices to prove that the morphism \[{\Bbb S}^{\cal J}[M \boxtimes F_{({\bf 0, 0})}^{\cal J}(B(M^{\rm gp}_{h{\cal J}}))] \xrightarrow{} {\Bbb S}^{\cal J}[M]\] induces an isomorphism on $\pi_0$. For this we use that ${\Bbb S}^{\cal J}[M \boxtimes F_{({\bf 0, 0})}^{\cal J}(B(M^{\rm gp}_{h{\cal J}}))] \cong {\Bbb S}^{\cal J}[M] \wedge {\Bbb S}^{\cal J}[F_{({\bf 0, 0})}^{\cal J}(B(M^{\rm gp}_{h{\cal J}}))]$ together with the fact that the composite functor ${\Bbb S}^{\cal J} \circ F_{({\bf 0, 0})}^{\cal J}$ equals the unreduced suspension functor $\Sigma_+^{\infty}$ to infer that the domain of the map in question is isomorphic to ${\Bbb S}^{\cal J}[M] \wedge \Sigma^{\infty}_+(B(M^{\rm gp}_{h{\cal J}}))$. We again apply the isomorphism \[\pi_0({\Bbb S}^{\cal J}[M] \wedge \Sigma^{\infty}_+(B(M^{\rm gp}_{h{\cal J}})) \cong \pi_0{\Bbb S}^{\cal J}[M] \otimes_{\Bbb Z} \pi_0\Sigma^{\infty}_+(B(M^{\rm gp}_{h{\cal J}}))\] of (\ref{pizeroiso}). Since $B(M^{\rm gp}_{h{\cal J}})$ is path-connected, this is isomorphic to $\pi_0{\Bbb S}^{\cal J}[M]$. 
\end{proof}

\begin{remark} It is easy to check that the commutative ${\cal J}$-space monoid $V(M)$ appearing in the above proof is ${\cal J}$-equivalent to $W(B^{\rm cy}(M))$ as defined in Definition \ref{w}. The former formulation is more convenient when we model the replete bar construction by the homotopy pullback $B^{\rm rep}(M)$ as opposed to the relative fibrant replacement $B^{\rm cy}(M)^{\rm rep}$. The advantage of the formulation given in Definition \ref{w} is that the resulting commutative ${\cal J}$-space monoid has a direct augmentation to the initial commutative ${\cal J}$-space monoid $U^{\cal J}$, which we used on numerous occasions in Section \ref{logtaq}.  
\end{remark}

\begin{proposition}\label{prop:logtaqetaleimplieslogetaledescent} Let $(R, P) \xrightarrow{(f, f^\flat)} (A, M)$ be a cofibration of cofibrant and connective pre-logarithmic ring spectra. If the $A$-module spectrum ${\rm TAQ}^{(R, P)}(A, M)$ is contractible, then $(f, f^\flat)$ satisfies log \'etale descent. That is, the natural map \[A \wedge_R {\rm THH}(R, P) \xrightarrow{\simeq} {\rm THH}(A, M)\] is a stable equivalence of commutative symmetric ring spectra.
\end{proposition}

\begin{proof} By Proposition \ref{logaugmentation}, we have that Theorem \ref{kuhnanalysis} is applicable to both ${\rm THH}(R, P)$ and ${\rm THH}(A, M)$. The isomorphism (\ref{pizeroiso}) then shows that it is also applicable to the augmented commutative $A$-algebra $A \wedge_R {\rm THH}(R, P)$.  We apply Theorem \ref{kuhnanalysis} to the morphism \[A \wedge_R {\rm THH}(R, P) \to {\rm THH}(A, M)\] of commutative augmented $A$-algebras to obtain a commutative diagram \[\begin{tikzcd}[row sep = small]A \wedge_R {\rm THH}(R, P) \ar{r} \ar{d}{\simeq} & {\rm THH}(A, M) \ar{d}{\simeq} \\ {\rm holim}(P_{A, n}(A \wedge_R {\rm THH}(R, P))) \ar{r} & {\rm holim}(P_{A, n}({\rm THH}(A, M)))\end{tikzcd}\] of commutative augmented $A$-algebras.

We claim that each of the morphisms \[P_{A, n}(A \wedge_R {\rm THH}(R, P)) \to P_{A, n}({\rm THH}(A, M))\] is a stable equivalence, from which it follows that the map on homotopy limits is also a stable equivalence. 

We proceed by induction on $n$. For $n = 0$, we consider the commutative diagram \[\begin{tikzcd}[row sep = small]A \wedge_R {\rm THH}(R, P) \ar{r} \ar{d} & {\rm THH}(A, M) \ar{d} \\ P_{A, 0}(A \wedge_R {\rm THH}(R, P)) \ar{r} & P_{A, 0}({\rm THH}(A, M)) \end{tikzcd}\] of augmented commutative $A$-algebras. The lower map is a stable equivalence by property (1) of Theorem \ref{kuhnanalysis}: the vertical maps in the diagram are stably equivalent to the augmentations to $A$, and the upper horizontal map is one over $A$. 

For $n > 0$, there is a map of fiber sequences 
\[\begin{tikzcd}[row sep = small]{[}\bigwedge_A^n {\rm taq}^A(A \wedge_R {\rm THH}(R, P)){]}_{h\Sigma_n} \ar{r} \ar{d} & {[}\bigwedge_A^n {\rm taq}^A({\rm THH}(A, M)){]}_{h\Sigma_n} \ar{d} \\ P_{A, n}(A \wedge_R {\rm THH}(R, P)) \ar{d} \ar{r} & P_{A, n}({\rm THH}(A, M)) \ar{d} \\  P_{A, n - 1}(A \wedge_R {\rm THH}(R, P)) \ar{r} & P_{A, n - 1}({\rm THH}(A, M)) \end{tikzcd}\] in the category of $A$-modules. By induction hypothesis, the bottom map is a stable equivalence. We prove that the morphism \begin{equation}\label{noteonproof}{\rm taq}^A(A \wedge_R {\rm THH}(R, P)) \xrightarrow{} {\rm taq}^A({\rm THH}(A, M))\end{equation} is a stable equivalence, which will conclude the proof. Its homotopy cofiber is ${\rm taq}^A(-)$ of the homotopy cofiber of the map $A \wedge_R {\rm THH}(R, P) \xrightarrow{} {\rm THH}(A, M)$ in the category of augmented commutative $A$-algebras. The latter homotopy cofiber is \[A \wedge^{\Bbb L}_{A \wedge_R {\rm THH}(R, P)} {\rm THH}(A, M) \cong (A \wedge_R R) \wedge^{\Bbb L}_{A \wedge_R {\rm THH}(R, P)} (A \wedge_A {\rm THH}(A, M)).\] By commuting homotopy pushouts, this is stably equivalent to the augmented commutative $A$-algebra $R \wedge_{{\rm THH}(R, P)}^{\Bbb L} {\rm THH}(A, M)$, which is stably equivalent to ${\rm THH}^{(R, P)}(A, M)$ by Proposition \ref{logthhjuggling}. By Lemma \ref{taqoflogthh}, the homotopy cofiber of (\ref{noteonproof}) is therefore stably equivalent to $\Sigma {\rm TAQ}^{(R, P)}(A, M)$, which is contractible by assumption. This concludes the proof.  \end{proof}

\begin{remark} One can also prove that (\ref{noteonproof}) is a stable equivalence by means of the homotopy cofiber sequence \[A \wedge_R \Sigma {\rm TAQ}(R, P) \xrightarrow{} \Sigma {\rm TAQ}(A, M) \xrightarrow{} \Sigma {\rm TAQ}^{(R, P)}(A, M)\] of Proposition \ref{logtransitivity}, since by Lemma \ref{lem:changeofrings} we have that Lemma \ref{taqoflogthh} is applicable to (\ref{noteonproof}).  The lack of naturality in Proposition \ref{logthhsuspension} may be dealt with by forming the relevant factorizations and lifts in a model category of arrows to make the construction natural in the morphism $(R, P) \to (A, M)$. 
\end{remark}

We have now provided a full proof of Theorem \ref{mainthm}, which we summarize here:

\begin{proof}[Proof of Theorem \ref{mainthm}] This follows from Propositions \ref{logetaledescentimplieslogthhetale}, \ref{logthhetaleimplieslogtaqetale} and \ref{prop:logtaqetaleimplieslogetaledescent}. 
\end{proof}

\section{Logarithmic ${\rm THH}$ of discrete log rings}\label{sec:discrete} By definition, a \emph{discrete pre-log ring} $(R, P, \beta)$ consists of a commutative ring $R$, a commutative monoid $P$ and a map $\beta \colon P \to (R, \cdot)$ of commutative monoids to the underlying multiplicative monoid of $R$. As explained in \cite[Section 5]{RSS15}, this gives rise to a pre-log ring spectrum $(HR, FP)$. A special case of this is Example \ref{ex:prelogstr}(3), which will be the main example of interest in this section. We shall simplify notation and write \[{\rm THH}^{(R, P)}(A, M) := {\rm THH}^{(HR, FP)}((HA, FM)^{\rm cof})\] for a map $(R, P) \to (A, M)$ of discrete pre-log rings. Here $(HA, FM)^{\rm cof}$ denotes a cofibrant replacement of $(HA, FM)$ relative to $(HR, FP)$ in the projective model structure. If $M$ is a discrete commutative monoid, we shall write $M^{\rm gp}$ for its usual group completion. 

The aim of this section is to prove the following: 

\begin{theorem}\label{thm:logthhbasechangedvrs} Let $R \to A$ be a tamely ramified finite extension of complete discrete valuation rings in mixed characteristic $(0, p)$ with perfect residue fields. Let $K \to L$ be the induced map of fraction fields. Then the natural map \[A \wedge_R {\rm THH}(R, R \cap {\rm GL}_1(K)) \to {\rm THH}(A, A \cap {\rm GL}_1(L))\] is a stable equivalence, i.e., the map $(R, R \cap {\rm GL}_1(K)) \to (A, A \cap {\rm GL}_1(L))$ of discrete log rings satisfies log \'etale descent. 
\end{theorem}

In fact, the conclusion of the Theorem \ref{thm:logthhbasechangedvrs} holds for any map which is \emph{log \'etale} in the following sense:

\begin{definition}\label{def:prelogetale} A map $(R, P^a) \to (A, M^a)$ of discrete log rings is \emph{log \'etale} if it arises as the logification of a map $(R, P) \to (A, M)$ of pre-log rings such that \begin{enumerate} \item the map $R \otimes_{{\Bbb Z}[P]} {\Bbb Z}[M] \to A$ of commutative rings is \'etale; \item the map $P^{\rm gp} \to M^{\rm gp}$ is an injection with finite cokernel of order invertible in $A$.\end{enumerate}
\end{definition}

This is an adaptation of Kato's definition of log \'etale maps of log schemes \cite[3.3, 3.5.2]{Kat89}, and in particular any log \'etale map in the sense of Definition \ref{def:prelogetale} gives rise to a log \'etale map $({\rm Spec}(A), M^a) \to ({\rm Spec}(R), P^a)$ of log schemes. 

\begin{example}\label{ex:tamelyramified} Recall (from e.g.\ \cite[Chapter 2.5]{Ser}) that any complete discrete valuation ring $R$ in mixed characteristic $(0, p)$ with perfect residue field $k$ is of the form $W(k)[x]/\Phi(x)$, with $W(k)$ the $p$-typical Witt vectors on $k$ and $\Phi(x)$ an Eisenstein polynomial of degree the absolute ramification index $e_{R}$ of $R$.

 Following Hesselholt--Madsen \cite[Section 2.2]{HM03} and Rognes \cite[Example 4.32]{Rog09}, let $R \to A$ be a tamely ramified finite extension of such discrete valuation rings. Let $\pi_{R}$ denote a uniformizer of $R$.  As explained in \cite[Proof of Lemma 2.2.6]{HM03}, we may reduce to the case where the extension is of the form $R \to R[x]/(x^{{e_A}/e_{R}} - \pi_{R})$; indeed, any map $R \to A$ of this form factors as an unramified extension followed by a totally tamely ramified extension. It is now clear that the logification of the map \[(R, \langle \pi_R \rangle) \to (R[x]/(x^{{e_A}/e_{R}} - \pi_{R}), \langle x \rangle)\] is log \'etale, since $e_A/e_R$ is coprime to $p$ by the assumption that the extension is tamely ramified.
\end{example}

The following definition makes use of the functor $\gamma$ from commutative ${\cal J}$-space monoids to Segal $\Gamma$-spaces introduced in \cite[Section 3]{Sag16}. It has the property that $M \to N$ is a weak equivalence in the group completion model structure if $\gamma(M) \to \gamma(N)$ is a weak equivalence of $\Gamma$-spaces. 

\begin{definition}\label{def:logetalespectra} A map $(R, P^a) \to (A, M^a)$ of log ring spectra is \emph{log \'etale} if it arises as the logification of a cofibration $(R, P) \to (A, M)$ of cofibrant pre-log ring spectra such that

\begin{enumerate}
\item the induced map $R \wedge_{{\Bbb S}^{\cal J}[P]} {\Bbb S}^{\cal J}[M] \to A$ is \'etale;
\item the $A$-module $A \wedge (\gamma(M)/\gamma(P))$ is contractible. 
\end{enumerate} 
\end{definition}

We only make use of the functor $\gamma$ to make the analogy with Definition \ref{def:prelogetale} as transparent as possible. In practice, we shall use the following description of the $A$-module $A \wedge (\gamma(M)/\gamma(P))$:

\begin{lemma}\label{lem:taqlemma} Let $(R, P) \to (A, M)$ be a cofibration of cofibrant pre-log ring spectra. There is a natural weak equivalence \[A \wedge (\gamma(M)/\gamma(P)) \simeq {\rm taq}^A(A \wedge_{{\Bbb S}^{\cal J}[M]} {\Bbb S}^{\cal J}[(M \boxtimes_P M)^{\rm rep}])\] of $A$-modules.
\end{lemma}

\begin{proof} By \cite[Proposition 5.19]{Sag14} and Theorem \ref{thm:taqcorepresents}, both of the $A$-modules in question corepresents the same functor as $A \wedge_{{\Bbb S}^{\cal J}[M]} {\rm TAQ}^{({\Bbb S}^{\cal J}[P], P)}({\Bbb S}^{\cal J}[M], M)$, namely the space of monoid derivations $X \mapsto {\rm Map}_{{\cal C}{\cal S}^{\cal J}_{P//M}}(M, (M + X)^{\cal J})$.
\end{proof}

A discrete commutative monoid $P$ is \emph{integral} if the canonical map $P \to P^{\rm gp}$ is injective, and a discrete pre-log ring $(R, P)$ is \emph{integral} if the underlying commutative monoid $P$ is.

\begin{proposition}\label{prop:discretelogetaleimplieslogetale} Let $(R, P^a) \to (A, M^a)$ be a log \'etale morphism of log rings which arises as the logification of a map $(f, f^\flat) \colon (R, P) \to (A, M)$ of integral pre-log rings satisfying the conditions of Definition \ref{def:prelogetale}. Assume moreover that ${\Bbb Z}[f^{\flat}] \colon {\Bbb Z}[P] \to {\Bbb Z}[M]$ is flat. Then the map $(Hf, Ff^\flat) \colon (HR, FP) \to (HA, FM)$, up to cofibrant replacements, logifies to a log \'etale map of log ring spectra. 
\end{proposition}

\begin{proof} We first check that the map $HR \wedge_{{\Bbb S}[P]}^{\Bbb L} {\Bbb S}[M] \to HA$ of commutative ring spectra is \'etale. On $\pi_0$ this is the map $R \otimes_{{\Bbb Z}[P]} {\Bbb Z}[M] \to A$ which is \'etale by assumption. Since both squares in the diagram \[\begin{tikzcd}{\Bbb S}[P] \ar{d} \ar{r} & H{\Bbb Z}[P] \ar{d} \ar{r} & HR \ar{d} \\ {\Bbb S}[M] \ar{r} & H{\Bbb Z}[M] \ar{r} & HR \wedge^{\Bbb L}_{H{\Bbb Z}[P]} H{\Bbb Z}[M]\end{tikzcd}\] are homotopy cocartesian, we obtain the desired isomorphism \[A \otimes_{R \otimes_{{\Bbb Z}[P]} {\Bbb Z}[M]} \pi_*(HR \wedge_{{\Bbb S}[P]}^{\Bbb L} {\Bbb S}[M]) \to A\] of graded rings from the flatness hypothesis.

We now check that the $HA$-module $HA \wedge^{\Bbb L} (\gamma(FM)/\gamma(FP))$ is contractible.  Recall that ${\Bbb S}^{\cal J} \circ F = \Sigma^{\infty}_+ = {\Bbb S}[-]$ is the usual spherical monoid ring construction, so that Lemma \ref{lem:taqlemma} asserts that this is naturally weakly equivalent to the $HA$-module \[{\rm taq}^{HA}(HA \wedge_{{\Bbb S}[M]}^{\Bbb L} {\Bbb S}^{\cal J}[(FM \boxtimes_{FP} FM)^{\rm rep}])\simeq {\rm taq}^{HA}(HA \wedge_{{\Bbb S}[M]}^{\Bbb L} {\Bbb S}^{\cal J}[F(M \oplus_{P} M)^{\rm rep}])\] Here $F(M \oplus_{P} M)^{\rm rep}$ denotes the repletion of the augmentation map $F(M \oplus_{P} M) \to F(M)$. Arguing as in \cite[Lemma 5.1]{RSS15}, we find that ${\Bbb S}^{\cal J}[F(M \oplus_{P} M)^{\rm rep}] \simeq {\Bbb S}[(M \oplus_{P} M)^{\rm rep}]$. 

We now apply \cite[Proposition 4.2.19]{Ogu18}, which asserts that $(M \oplus_{P} M)^{\rm rep} \cong M \oplus M^{\rm gp}/P^{\rm gp}$ (this uses the integrality hypothesis). This allows us to further rewrite \[HA \wedge_{{\Bbb S}[M]}^{\Bbb L} {\Bbb S}[(M \oplus_{P} M)^{\rm rep}] \simeq HA \wedge_{{\Bbb S}[M]}^{\Bbb L} {\Bbb S}[M \oplus M^{\rm gp}/P^{\rm gp}] \simeq HA[M^{\rm gp}/P^{\rm gp}].\] We can compute the ${\rm taq}$ of this group ring using $\Gamma$-homology. Notice that Lemma \ref{lem:changeofrings} implies that ${\rm taq}^{HA}(HA[M^{\rm gp}/P^{\rm gp}])$ is naturally weakly equivalent to \[{\rm taq}^{HA[M^{\rm gp}/P^{\rm gp}]}(HA[M^{\rm gp}/P^{\rm gp}] \wedge_{HA} HA[M^{\rm gp}/P^{\rm gp}]) \wedge_{HA[M^{\rm gp}/P^{\rm gp}]}^{\Bbb L} HA,\] which is typically denoted ${\rm TAQ}^{HA}(HA[M^{\rm gp}/P^{\rm gp}] ; HA)$. By Basterra--McCarthy \cite{BM02} (see also Basterra--Richter \cite[Section 2]{BR04}), the homotopy groups of this ${\rm TAQ}$-term is equivalent to the $\Gamma$-homology $H\Gamma_*^{HA}(HA[M^{\rm gp}/P^{\rm gp}] ; HA)$. By Richter--Robinson \cite[Proposition 3.1]{RR04}, this is $(HA)_*(H(M^{\rm gp}/P^{\rm gp}))$; the $HA$-homology of $H(M^{\rm gp}/P^{\rm gp})$. This is trivial since the order of $M^{\rm gp}/P^{\rm gp}$ is invertible in $A$. 
\end{proof}

Together with Proposition \ref{prop:discretelogetaleimplieslogetale}, the following connects Definition \ref{def:prelogetale} to the notions of formal log \'etaleness studied in Section \ref{lastsection}: 

\begin{lemma}\label{lem:logetaleimpliesformallylogetale} Let $(f, f^\flat) \colon (R, P) \to (A, M)$ be a log \'etale morphism of log ring spectra. Then $(f, f^\flat)$ is formally log ${\rm TAQ}$-\'etale.
\end{lemma}

\begin{proof} By Proposition \ref{prop:logtaqlogi}, we may without loss of generality assume that $(f, f^\flat)$ is a map of pre-log ring spectra satisfying the hypotheses of Definition \ref{def:logetalespectra}. By \cite[Lemma 11.25]{Rog09} or \cite[Lemma 6.2]{Sag14} there is a homotopy cofiber sequence \begin{equation}\label{cofiberseqquot}A \wedge (\gamma(M)/\gamma(P)) \to {\rm TAQ}^{(R, P)}(A, M) \to {\rm TAQ}^{R \wedge_{{\Bbb S}^{\cal J}[P]} {\Bbb S}^{\cal J}[M]}(A)\end{equation} of $A$-modules. The first term in this sequence is contractible by assumption, while the right-hand ${\rm TAQ}$-term is contractible as a consequence of $R \wedge_{{\Bbb S}^{\cal J}[P]} {\Bbb S}^{\cal J}[M] \to A$ being \'etale (see e.g.\ \cite[Corollary 7.5.4.5]{Lur}).
\end{proof}

\begin{remark} At the time of writing, it is not clear to the author under which hypotheses formally log ${\rm TAQ}$-\'etale implies log \'etale in the sense of Definition \ref{def:logetalespectra}. This is closely related to \cite[Remark 11.26]{Rog09}. 
\end{remark}

\begin{proof}[Proof of Theorem \ref{thm:logthhbasechangedvrs}] Example \ref{ex:tamelyramified} shows that any such map of discrete valuation rings is log \'etale. Proposition \ref{prop:discretelogetaleimplieslogetale} shows that this gives rise to a log \'etale map of log ring spectra. By Lemma \ref{lem:logetaleimpliesformallylogetale}, any log \'etale map is formally log ${\rm TAQ}$-\'etale. Since Eilenberg--Mac\,Lane spectra and spherical monoid rings are connective, Theorem \ref{mainthm} applies to conclude the proof.
\end{proof}

\section{Logarithmic ${\rm TAQ}$ as a cotangent complex}\label{sec:cotangentcomplex} In this short final section, we explain how our description of log ${\rm TAQ}$ may be interpreted in the cotangent complex formalism of Lurie \cite[Section 7]{Lur}. We shall continue to work in the context of model categories, and we refer to the work of Harpaz--Nuiten--Prasma \cite{HNP19} for a construction of the tangent bundle in this context. 

In \cite[Section 3]{SSV16}, a \emph{replete} model structure on the category ${\rm PreLog}$ of pre-logarithmic ring spectra is described. The fact that \emph{loc.\ cit.} works with simplicial pre-log rings does not give rise to technical difficulties, and their arguments apply \emph{mutatis mutandis} in the present context. This model structure arises by forming a left Bousfield localization of the projective model structure on ${\rm PreLog}$ with respect to the set $({\Bbb S}^{\cal J}[Q], Q)$, where $Q$ is the set of morphisms in ${\cal C}{\cal S}^{\cal J}$ at which one localizes to obtain the group completion model structure. Consequently, for a morphism $(B, N) \to (A, M)$, the map $(B, N) \xrightarrow{} (B \wedge_{{\Bbb S}^{\cal J}[N]} {\Bbb S}^{\cal J}[N^{\rm rep}], N^{\rm rep})$ is an acyclic cofibration in this model structure. However, there is no reason to believe that its codomain is fibrant over $(A, M)$. Here we explain how to alleviate this issue by only forming the left Bousfield localization after passing to the appropriate slice category:

\begin{proposition}\label{prop:repletemodel} Let $(A, M)$ be a cofibrant pre-logarithmic ring spectrum. The category ${\rm PreLog}_{(A, M)//(A, M)}$ admits a \emph{replete} model structure ${\rm PreLog}^{\rm rep}_{(A, M)//(A, M)}$ in which the fibrant objects are the fibrant $(B, N) \to (A, M)$ with $N$ replete over~$M$. Up to a fibrant replacement in the projective model structure, the functor \[((B, N) \to (A, M)) \mapsto ((B \wedge_{{\Bbb S}^{\cal J}[N]} {\Bbb S}^{\cal J}[N^{\rm rep}], N^{\rm rep}) \to (A, M))\] is a fibrant replacement functor in this model structure. 
\end{proposition}

\begin{proof} Let $Q(M)$ be a set of generating acyclic cofibrations in the slice category $({\cal C}{\cal S}^{\cal J}_{\rm gp})_{M // M}$. Define the replete model structure on ${\rm PreLog}_{(A, M)//(A, M)}$ to be the left Bousfield localization of the projective one with respect to the set \[S = (A \wedge_{{\Bbb S}^{\cal J}[M]} {\Bbb S}^{\cal J}[Q(M)], Q(M)).\] Since ${\rm Map}_{{\rm PreLog}_{(A, M)//(A, M)}}((A \wedge_{{\Bbb S}^{\cal J}[M]} {\Bbb S}^{\cal J}[Q(M)], Q(M)), (B, N))$ is naturally weakly equivalent to  ${\rm Map}_{{\cal C}{\cal S}^{\cal J}_{M // M}}(Q(M), N)$ by adjunction, we find that $(B, N)$ is S-local if and only if $N$ is $Q(M)$-local. This means that $N$ is fibrant in $({\cal C}{\cal S}^{\cal J}_{\rm gp})_{M // M}$, i.e., replete over $M$. 

The proposed fibrant replacement functor arises from considering the pushout of the diagram \[(B, N) \xleftarrow{} (A \wedge_{{\Bbb S}^{\cal J}[M]} {\Bbb S}^{\cal J}[N], N) \xrightarrow{} (A \wedge_{{\Bbb S}^{\cal J}[M]} {\Bbb S}^{\cal J}[N^{\rm rep}], N^{\rm rep})\] of pre-logarithmic ring spectra. 
\end{proof}

In \cite[Remark 8.8]{Rog09}, a version of the stabilization ${\rm Sp}({\rm PreLog}^{\rm rep}_{(A, M)//(A, M)})$ is suggested as a category of \emph{log modules} ${\rm Mod}_{(A, M)}$. We point out two desirable properties of this category here:

\begin{enumerate}
\item We can realize log ${\rm THH}$ as the cyclic bar construction in ${\rm PreLog}$, provided that we pass to the replete model structure after forming it. Since $B^{\rm cy}_{(R, P)}(A, M) = (B^{\rm cy}_R(A), B^{\rm cy}_P(M))$, we have that \[B^{\rm cy}_{(R, P)}(A, M)^{\rm rep} = (B^{\rm cy}_R(A) \wedge_{{\Bbb S}^{\cal J}[B^{\rm cy}_P(M)]} {\Bbb S}^{\cal J}[B^{\rm cy}_P(M)^{\rm rep}], B^{\rm cy}_P(M)^{\rm rep})\] is a fibrant replacement of the cyclic bar construction over $(A, M)$. This is $({\rm THH}^{(R, P)}(A, M), B^{\rm cy}_P(M)^{\rm rep})$ by definition. 
\item The chain of Quillen equivalences (\ref{basterramandelleq}) exhibit ${\rm Mod}_A$ as the tangent category of ${\cal C}{\rm Sp}^{\Sigma}$ at $A$ and ${\rm TAQ}^R(A)$ as its cotangent complex. By definition, the stabilized category ${\rm Sp}({\rm PreLog}_{(A, M)//(A, M)})$ is the \emph{tangent category} of the category of pre-logarithmic ring spectra at $(A, M)$. This comes with a \emph{cotangent complex} ${\Bbb L}_{(A, M) | (R, P)} = \Sigma^{\infty}(A \wedge_R A, M \boxtimes_P M)$. If we pass to the replete model structure before stabilizing, we find that the cotangent complex ${\Bbb L}_{(A, M) | (R, P)}$ is level equivalent to \[ \Sigma^{\infty}((A \wedge_R A) \wedge_{{\Bbb S}^{\cal J}[M \boxtimes_P M]} {\Bbb S}^{\cal J}[(M \boxtimes_P M)^{\rm rep}], (M \boxtimes_P M)^{\rm rep}).\] We denote this \emph{replete} cotangent complex by ${\Bbb L}_{(A, M) | (R, P)}^{\rm rep}$. This comes with an underlying spectrum object of augmented $A$-algebras, which by the Quillen equivalences (\ref{basterramandelleq}) corresponds to ${\rm TAQ}^{(R, P)}(A, M)$ as we defined it in Definition \ref{def:logtaq}. This was our initial motivation for pursuing this description of log~${\rm TAQ}$. 
\end{enumerate}

\end{document}